\documentclass[10pt]{amsart}
\usepackage{amsmath,amssymb,amsthm}
\usepackage{graphicx}
\usepackage[english]{babel}
\usepackage[utf8]{inputenc}
\usepackage[T1]{fontenc}
\usepackage{mathtools}
\usepackage{bbm}
\usepackage[margin=3cm]{geometry}
\allowdisplaybreaks
\usepackage{xcolor}
\usepackage{enumitem}
\usepackage{hyperref}
\usepackage{float}
\usepackage{subcaption}

\newcommand{\N}{\mathbb{N}}
\newcommand{\Z}{\mathbb{Z}}

\newcommand{\R}{\mathbb{R}}

\newcommand{\E}{\mathbb{E}}
\newcommand{\dd}{\,\mathrm{d}}
\newcommand{\e}{{\mathrm{e}}}
\newcommand{\PP}{\mathbb{P}}
\newcommand{\F}{\mathcal{F}}

\newcommand{\re}{\!\restriction}

\newcommand{\Ev}{\mathcal{E}}

\DeclareMathOperator{\supp}{supp}

\numberwithin{equation}{section}
\numberwithin{figure}{section}

\newtheorem{theorem}{Theorem}[section]
\newtheorem{corollary}[theorem]{Corollary}

\newtheorem{lemma}[theorem]{Lemma}
\newtheorem{proposition}[theorem]{Proposition}
\theoremstyle{remark}
\newtheorem{remark}[theorem]{Remark}
\theoremstyle{definition}

\begin{document}

\title{Scaling limits for fractional polyharmonic Gaussian fields}

\author[N. De Nitti]{Nicola De Nitti}
\address[N. De Nitti]{EPFL, Institut de Mathématiques, Station 8, 1015 Lausanne, Switzerland.
 \emph{Email address: \texttt{nicola.denitti@epfl.ch}}.}

\author[F. Schweiger]{Florian Schweiger}
\address[F. Schweiger]{Weizmann Institute of Science, Department of Mathematics, 7610001 Rehovot, Israel.\\
Current address: Université de Genève, Section de mathématiques, Rue du Conseil 
G\'{e}n\'{e}ral 7--9, 1205 Gen\`{e}ve, Switzerland.\\
\emph{Email address: \texttt{florian.schweiger@unige.ch}}.}


\begin{abstract}
This work is concerned with fractional Gaussian fields, i.e. Gaussian fields whose covariance operator is given by the inverse fractional Laplacian $(-\Delta)^{-s}$ (where, in particular, we include the case $s >1$). We define a lattice discretization of these fields and show that their scaling limits -- with respect to the optimal Besov space topology (up to an endpoint case) -- are the original continuous fields. As a byproduct, in dimension $d<2s$, we prove the convergence in distribution of the maximum of the fields. A key tool in the proof is a sharp error estimate for the natural finite difference scheme for $(-\Delta)^s$ under minimal regularity assumptions, which is also of independent interest.
\end{abstract}


\maketitle
\tableofcontents
\section{Introduction}
\label{sec:intro}
\begin{figure}
		\subfloat[$s=0$ (white noise)]{\includegraphics[width = 0.3\textwidth]{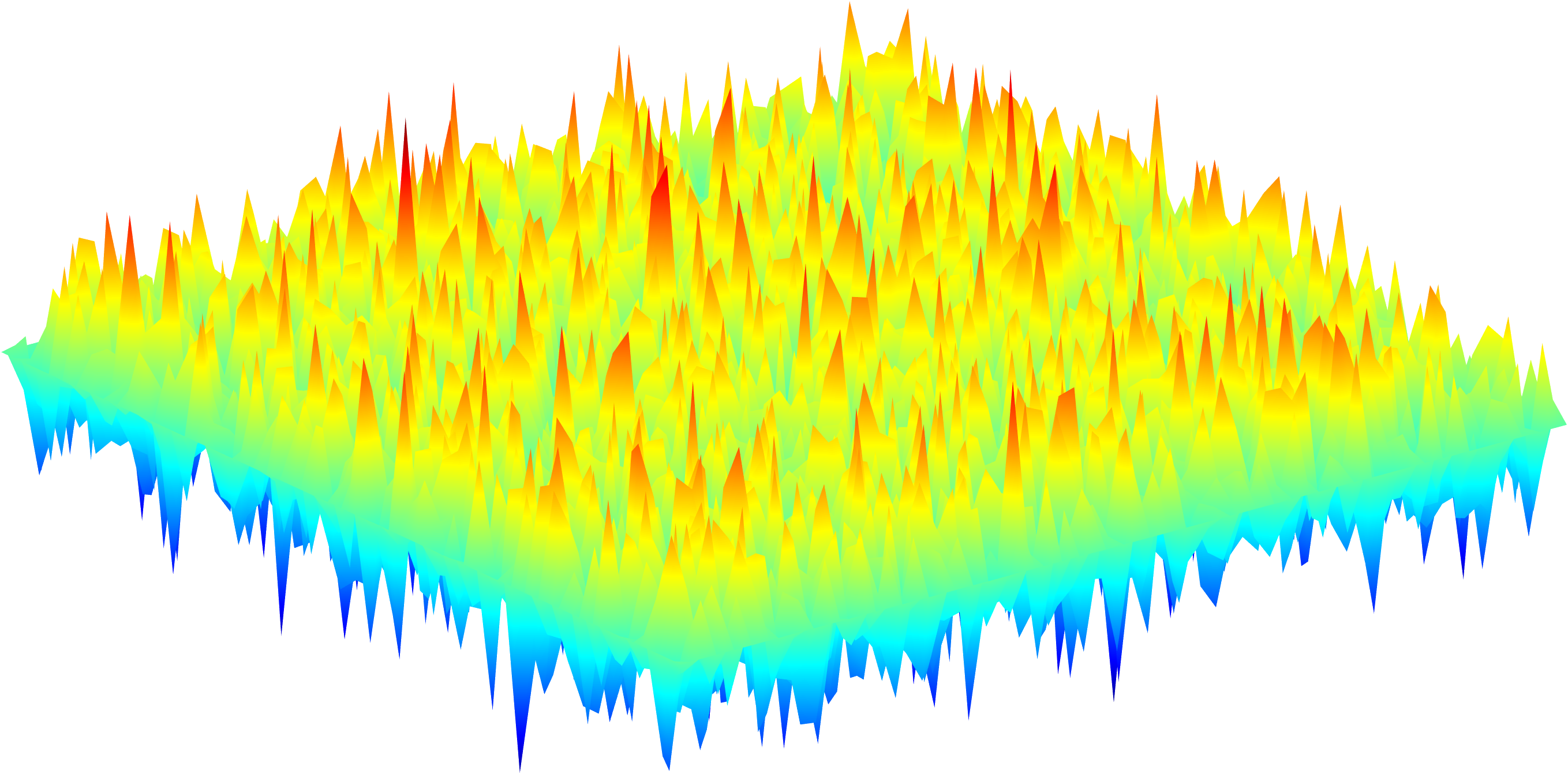}} 	\hfill
	\subfloat[$s=0.5$]{\includegraphics[width = 0.3\textwidth]{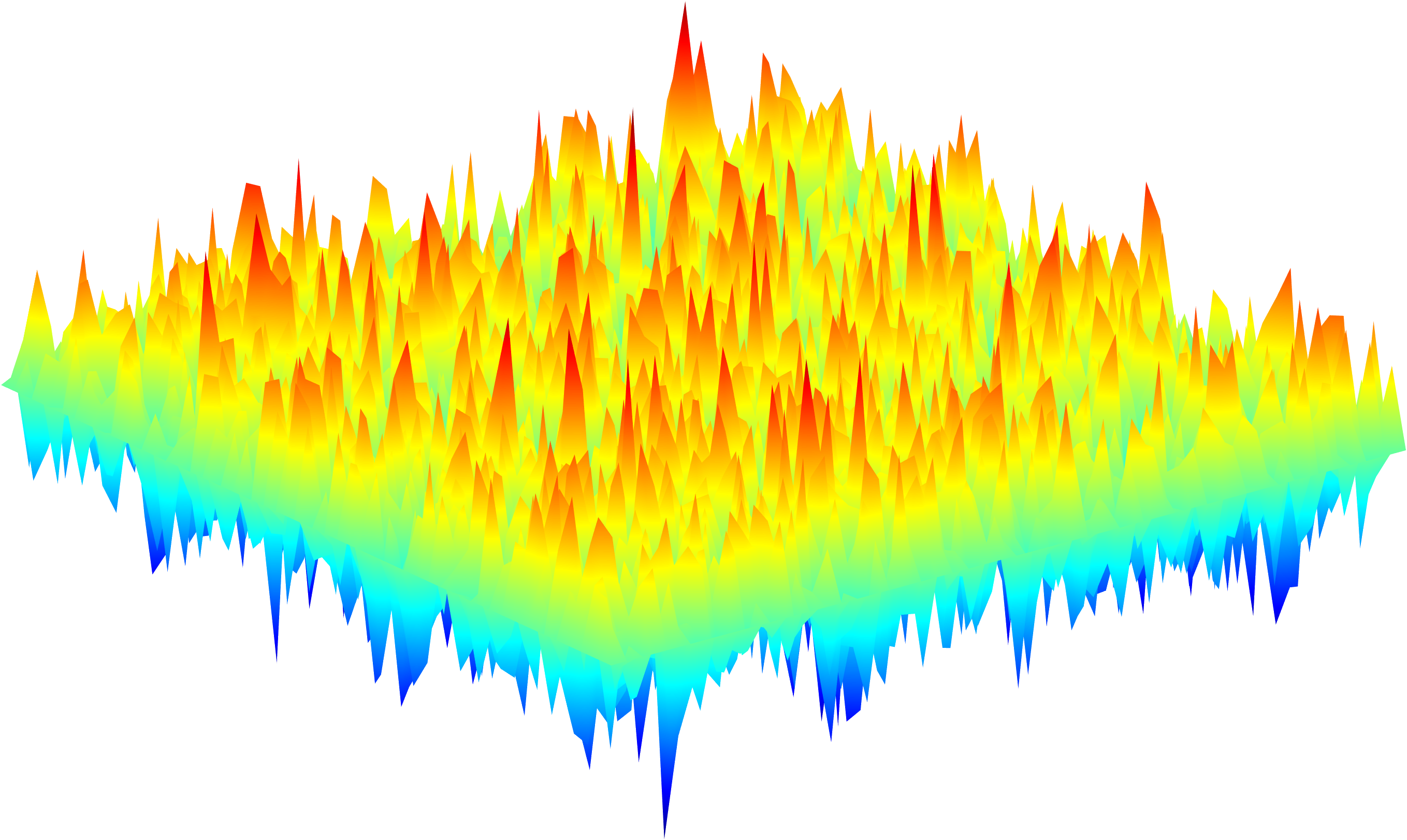}} 
	\hfill 
	\subfloat[$s=1$ (Gaussian free field)]{\includegraphics[width = 0.3\textwidth]{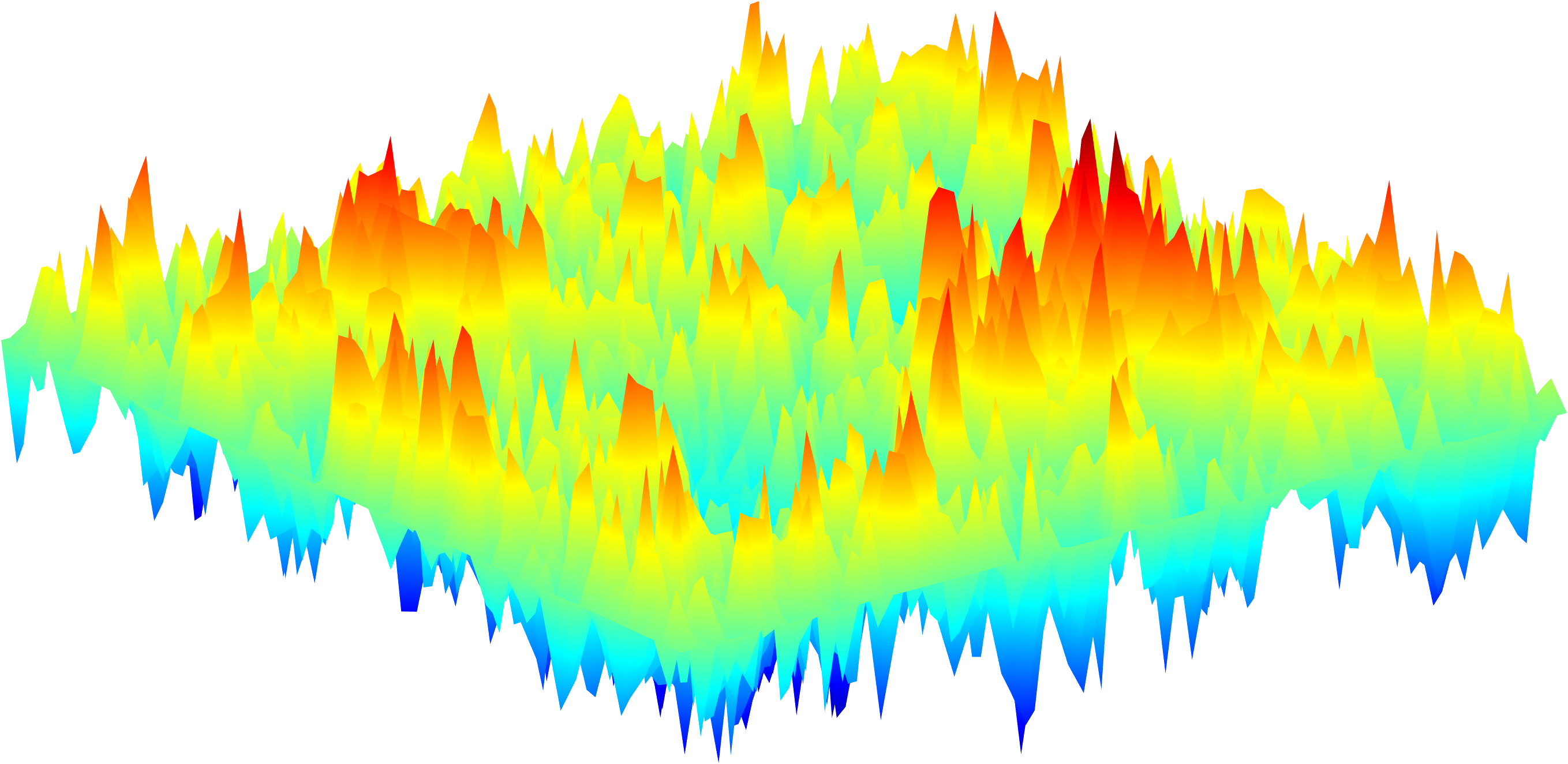}}
\\
	\subfloat[$s=1.5$]{\includegraphics[width =  0.3\textwidth]{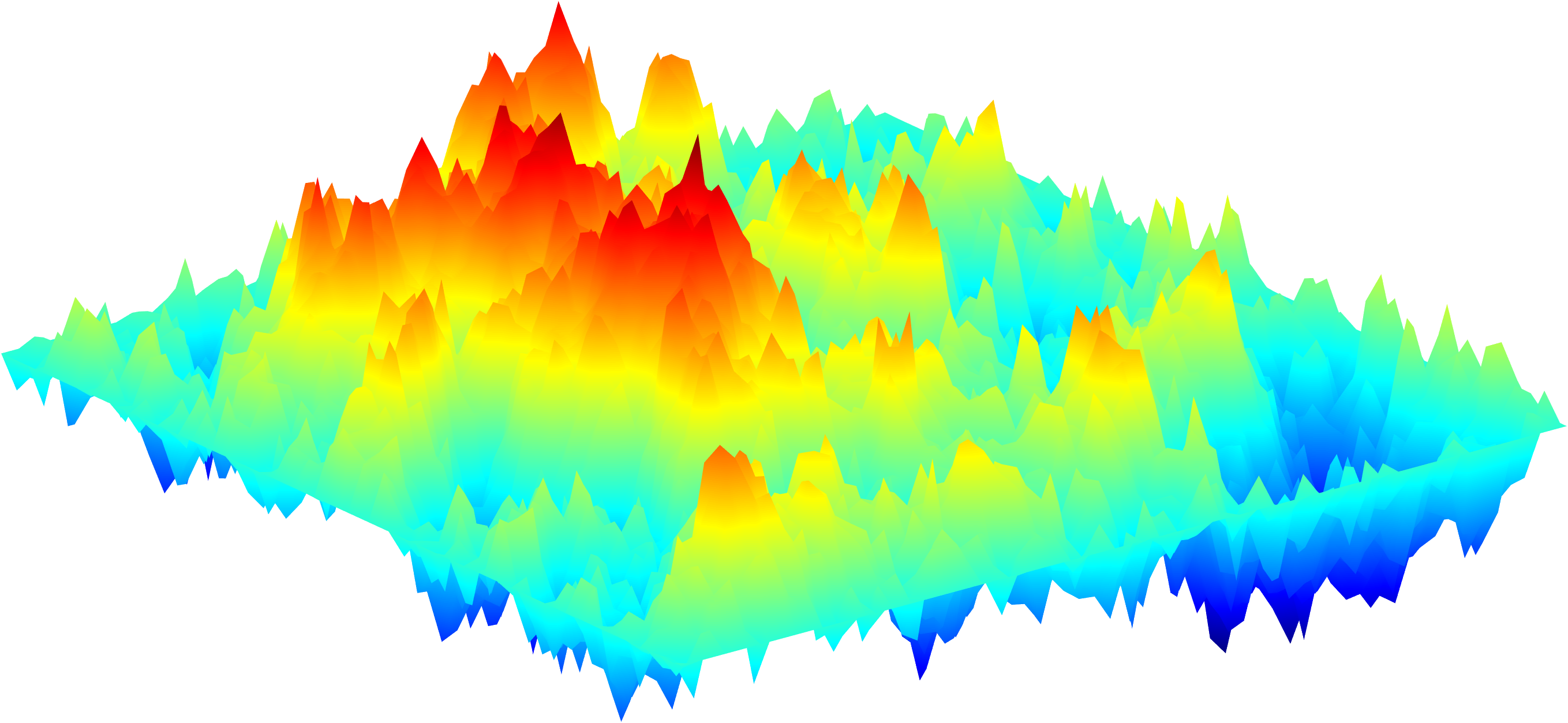}} 
	\hfill 
	\subfloat[$s=2$ (membrane model)]{\includegraphics[width = 0.3\textwidth]{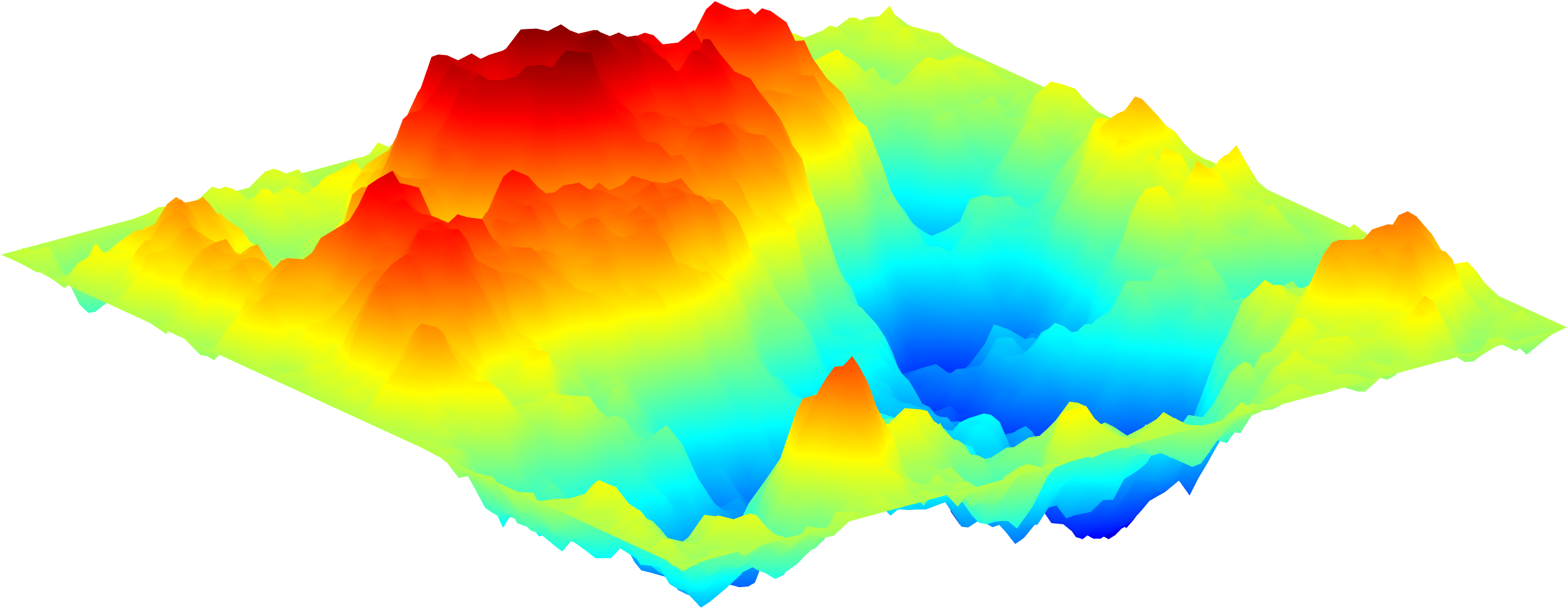}}
	\hfill 
	\subfloat[$s=2.5$]{\includegraphics[width = 0.3\textwidth]{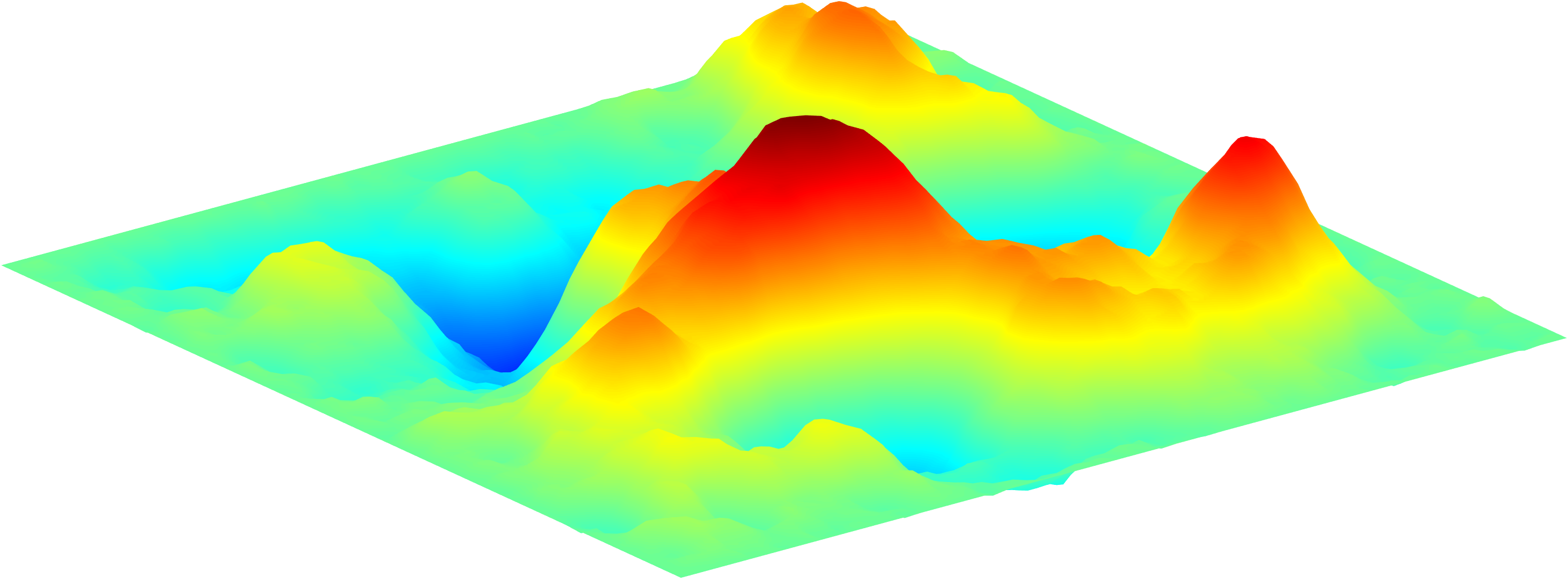}} 
\\
	\subfloat[$s=3$]{\includegraphics[width = 0.3\textwidth]{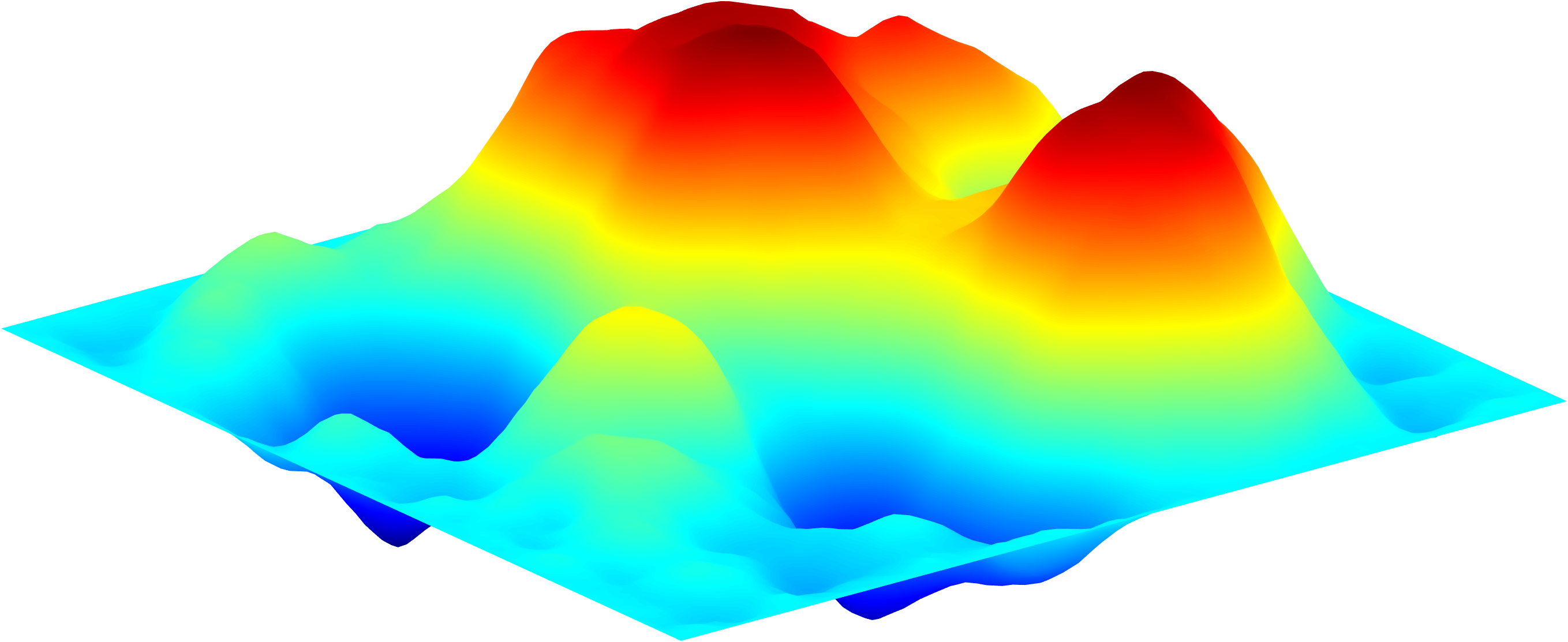}} 	\hfill
	\subfloat[$s=3.5$]{\includegraphics[width = 0.3\textwidth]{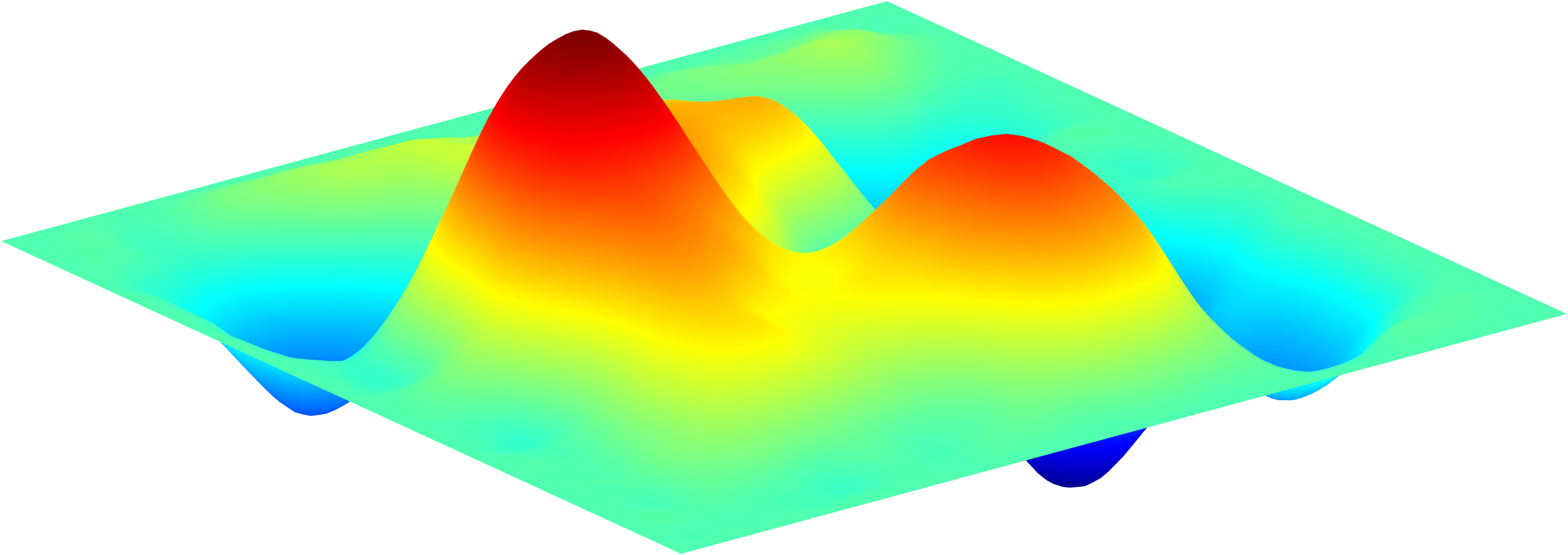}} 
	\hfill
	\subfloat[$s=4$]{\includegraphics[width = 0.3\textwidth]{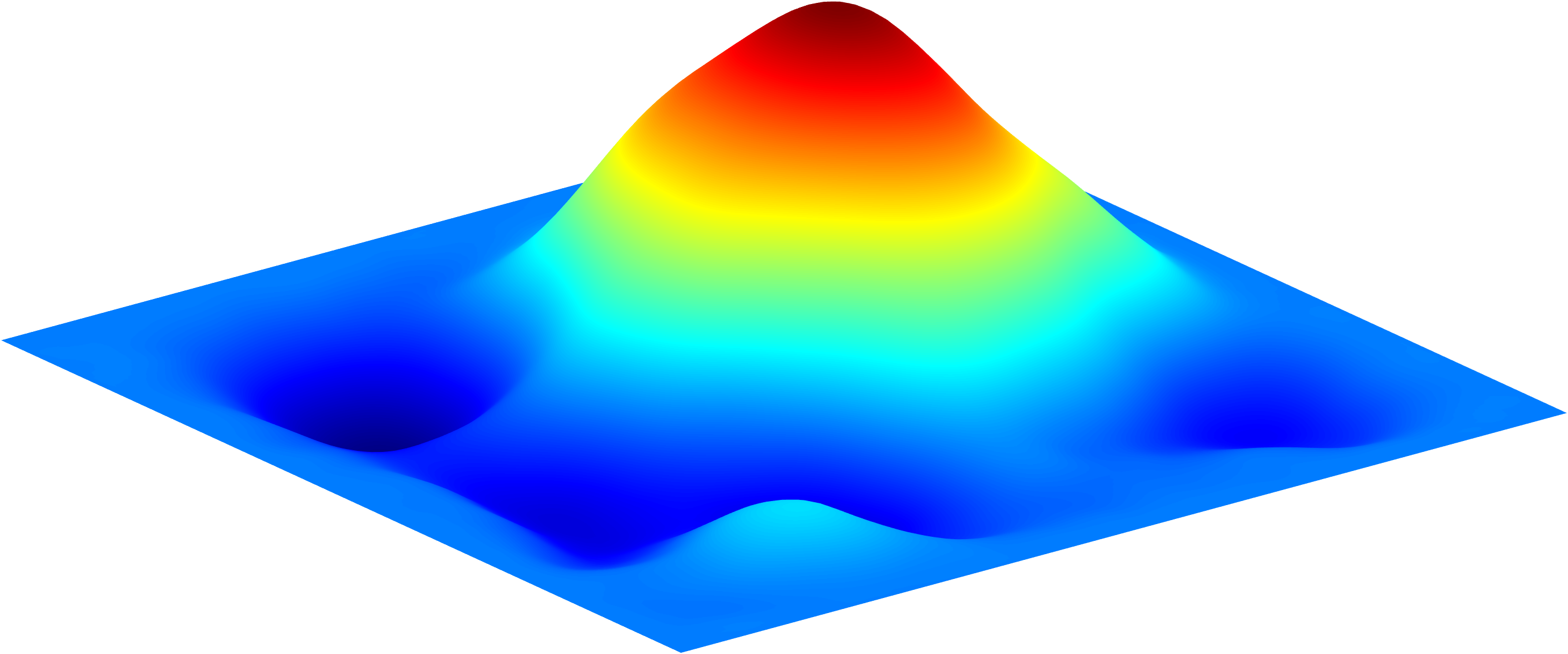}} 
	
	\caption{Surface plots of discrete fractional polyharmonic Gaussian fields on $\Omega \coloneqq  (0, 1)^2 \cap\frac{1}{100}\Z^2$ with zero boundary conditions. These discrete random functions are linearly interpolated. See also \cite[Fig. 1.1]{LSSW16} for further numerical experiments. }
	\label{fig:ex}
\end{figure}
\subsection{Fractional Gaussian Fields}\label{s:introFGF}
Fractional Gaussian fields (in short FGFs) form a natural one-parameter family of Gaussian interface models. For a fixed parameter $s\ge0$, the $s$-fractional Gaussian field is the Gaussian field whose covariance operator is $(-\Delta)^{-s}$, the inverse of the  fractional Laplacian of order $s$. We emphasize right away that we do not assume $s\in[0,1]$, and in fact our main interest is in the polyharmonic case $s>1$. A purely formal and non-rigorous way to define the $s$-fractional Gaussian field on some domain $\Omega\subset\R^d$ is to set
\begin{align}\label{eq:fgf}
\PP(\mathrm{d}\varphi)=\frac{1}{Z}\exp\left(-\frac12\int_\Omega\varphi(x)((-\Delta)^s\varphi)(x)\dd x\right)\dd\varphi.
\end{align}
This cannot be taken as a rigorous definition, as $\dd\varphi$ refers to the Lebesgue measure on the infinite-dimensional space $\R^\Omega$, which does not exist. 

There are also other issues with \eqref{eq:fgf}: namely, one needs to decide how to define $(-\Delta)^s$ for functions $\varphi\colon\Omega\to\R$ and (closely related to that issue) one needs to decide on boundary values of $\varphi$. For these questions, we have a clear answer, though. We take $0$ boundary values (i.e., we take $\varphi$ to be extended by $0$ to the whole $\R^d$), and we let $(-\Delta)^s$ be the fractional Laplacian on the full space $\R^d$, which is defined by using the Fourier transform\footnote{~An equivalent hypersingular integral formulation of the polyharmonic fractional operator of order $s  \in (0,m)$, for any \(m \in \mathbb{N}\),  is given by 
	\begin{align*}
	(-\Delta)^{s} u(x)\coloneqq C_{d, m, s} \int_{\mathbb{R}^{d}} \frac{\sum_{j=-m}^{m}(-1)^{j}\begin{pmatrix} 2 m \\ m-j\end{pmatrix} u(x+j y)}{|y|^{d+2 s}} \dd y,
	\end{align*}
	where
	\begin{align*}
	C_{d, m, s}\coloneqq  \begin{cases}\displaystyle \frac{2^{2 s} \Gamma(d / 2+s)}{\pi^{d / 2} \Gamma(-s)}\left(\sum_{j=1}^{m}(-1)^{j}\begin{pmatrix}
	2 m \\
	m-j
	\end{pmatrix} j^{2 s}\right)^{-1} & \text { if } s \in(0, m) \backslash \mathbb{N}, \\
	\displaystyle \frac{2^{2 s} \Gamma(d / 2+s) s !}{2 \pi^{d / 2}}\left(\sum_{j=2}^{m}(-1)^{j}\begin{pmatrix}
	2 m \\
	m-j
	\end{pmatrix} j^{2 s} \ln j\right)^{-1} & \text { if } s \in(0, m) \cap \mathbb{N} .\end{cases}
	\end{align*}
	
	We refer to \cite{Abatangelo_2021}  and references therein for further information on the theory of higher-order fractional Laplacians.}. That is, for any \(\xi \in \mathbb{R}^{d}\), 
\begin{align*}
\F\left[(-\Delta)^{s} u\right](\xi)=|\xi|^{2 s} \F[u](\xi)
\end{align*}
with \[\F[u](\xi)\coloneqq \int_{\R^d}\e^{-i\xi\cdot x}u(x)\dd x.\]
These choices are natural from a probabilistic point of view, as we will explain in Remark \ref{rk:prob}, and they can be implemented to provide a rigorous meaning to \eqref{eq:fgf}, for example, as a probability measure on the space of tempered distributions. This is discussed in detail in the excellent survey \cite{LSSW16} and in Section \ref{ssec:cont-FGF} we recall the points that are important for us.

When studying a continuum random field, it is useful to define a lattice-regularized version of it. This is particularly relevant in probabilistic approaches to quantum field theory, where often it is extremely hard to define the continuum random field in question. One possible strategy is to define the corresponding field on a lattice first (which is usually much easier), and then try to prove that one can take a scaling limit.

In the setting of FGFs, of course, it is known how to construct a continuum FGF. Nonetheless it is still natural to wonder how one should define a discrete version of the FGF and whether one can recover the continuum FGF as a scaling limit. We want to address this question in a unified way for all $s\ge0$.

Thus, the main goal of the present work is to define a discrete version $\varphi_h$ of the $s$-fractional Gaussian field $\varphi$ on a lattice $\Omega_h\coloneqq \Omega\cap h\Z^d$ and to show that in the limit $h\to0$ the discrete $s$-fractional Gaussian field converges in law, with respect to a suitable topology, to the (continuous) $s$-fractional Gaussian field. Our main result is that, with our definition of a $s$-fractional Gaussian field, this convergence holds in a rather strong sense, namely in law with respect to the topology of Besov spaces for the optimal range of parameters.

Similar problems have been studied before for specific values of $s$. If $s=1$, the field is the Gaussian free field and the convergence of the discrete Gaussian free field to its continuous variant is folklore (see \cite[Section 4]{S07} for related results). The proof relies on the fact that covariances of the discrete Gaussian free field can be represented using simple random walk, which in the scaling limit becomes Brownian motion. The case $0\le s<1$ is addressed in \cite[Section 12]{LSSW16}\footnote{~There, a definition of the discrete FGF that is slightly different from ours is used; the proof, however, should apply to all reasonable discretizations including ours. Moreover, the scaling in \cite[Section 12.2]{LSSW16} is incorrect, as we explain in Footnote \ref{fn:LSSW}.}, and the proof of the scaling limit follows a similar strategy as for the case $s=1$, just with the $2s$-stable L\'evy process taking the place of Brownian motion.

These results for $s\le1$ all rely on some form of a random walk representation. For $s>1$ and for our choice of boundary values, there is no such random walk representation and so proofs become much more difficult. However, if one  one uses another definition of the FGF in terms of spectral powers of the ordinary Laplacian (the so-called eigenfunction FGF from \cite[Section 9]{LSSW16}), one retains a random walk representations and it is comparably easy to establish a scaling limit. On the torus, the eigenfunction FGF agrees with the ordinary FGF, and for the discrete FGF on the torus results similar to ours have been shown in \cite{MR4282689, MR4166192} and very recently and independently in \cite{2302.02963}. One can also study the eigenfunction FGF on domains with boundary where it is genuinely different from the ordinary FGF. In fact, in \cite{BC22}, this is done not in the lattice case, but in the more complicated case of a Sierpinski gasket.

Let us emphasize again, though, that in our setting in the regime $s>1$ the presence of the zero boundary values adds genuine new difficulties. 
The only existing result in this regime is for $s=2$, where the $s$-fractional Gaussian field is the so-called membrane model. In \cite{CDH19}, it was proven that this field is the scaling limit of its discrete version. The main ingredient in the proof were estimates for finite difference schemes for $(-\Delta)^2$ from \cite{T64}, and estimates for its Green's function from \cite{MS19}.

Thus, previous work was restricted to $s\in[0,1]\cup\{2\}$, while our results cover the entire range $s\in[0,\infty)$. Even in the case $s\in[0,1]\cup\{2\}$, our results improve upon the previous work. Namely, the convergence in \cite[Section 12]{LSSW16} is with respect to the topology of distributions and the convergence in \cite{CDH19} is with respect to the topology of some negative Sobolev space (for non-optimal parameters). As an easy corollary of our result with respect to the Besov-space topology, one obtains convergence with respect to the Sobolev-topology and also with respect to the H\"older topology (both with the optimal range of parameters).

Our method of proof uses estimates for finite difference schemes like \cite{CDH19}, but of a different flavor. Instead of the estimate from \cite{T64} used in \cite{CDH19} (that needs $C^k$-regularity of the function to be approximated by the scheme), we establish an estimate that needs only minimal regularity assumptions (essentially just $H^{s+\varepsilon}$-regularity for some $\varepsilon>0$). This is the main technical result of the paper and we will discuss it and its context next.

\subsection{Finite difference schemes for fractional operators}\label{s:findifffrac}
There is a close relation between discrete versions of Gaussian fields and finite difference approximations of the corresponding operator. Indeed, if we want to define a lattice version of \eqref{eq:fgf} that is suitably close to \eqref{eq:fgf} itself, then we need a lattice approximation of $(-\Delta)^s$; the better this approximation, the closer the resulting lattice field will be to its continuous version.

Before discussing finite difference schemes, let us mention that there has been work on finite element approaches to the fractional Laplacian (at least for $s\le1$). We cannot cover the whole body of relevant literature here, but we refer to the very recent survey \cite{BLN22}.

Let us now turn to finite difference schemes. The subject of finite difference schemes for the fractional Laplacian $(-\Delta)^s$ has been studied before and various schemes have been proposed. However, the main focus has been on the case $s<1$ and often also $d=1$. We refer to the survey \cite{HO16} and the references therein for an overview. The main challenge when constructing a finite difference scheme for the fractional Laplacian is that it is given by convolution with a singular integral kernel, and a naive discretization of this kernel might not capture its behavior near the singularity.

Our preferred way to construct a finite difference scheme arises naturally when working in Fourier space. Let us consider first the usual Laplacian, i.e., the case $s=1$. Its symbol\footnote{~The symbol is defined as the Fourier multiplier corresponding to the operator in real space. Here, $|\xi|^2\F[u](\xi)=\F[(-\Delta) u](\xi)$.} is $|\xi|^2$ and its standard finite difference approximation (given by $-\Delta_hu_h=-\sum_{j=1}^d\frac{1}{h^2}\left(u_h(x+he_j)+u_h(x-he_j)-2u_h(x)\right)$ in dimension $d$) has symbol 
\begin{equation}\label{e:symbolLap}
M_h(\xi)^2\coloneqq \sum_{j=1}^d\frac{4}{h^2}\sin^2\left(\frac{\xi_j h}{2}\right).
\end{equation}
So, for the fractional Laplacian $(-\Delta)^s$ with symbol $|\xi|^{2s}$, a natural way to define a finite difference scheme is to take the finite difference operator with symbol $M_h(\xi)^{2s}$.
That is, we define 
\[
\F_h\left[(-\Delta_h)^{s} u_h\right](\xi)=M_h(\xi)^{2s}\F_h[u_h](\xi),
\]
with 
\[\F_h[u_h](\xi)\coloneqq h^d\sum_{x\in h\Z^d}\e^{-i\xi\cdot x}u_h(x),\]
for $u_h: h\Z^d \to \R$. We can also use $M_h(\xi)^{2s}$ as a continuous Fourier multiplier and thereby understand $(-\Delta_h)^s$ also as a continuous operator, defined such that $\F\left[(-\Delta_h)^{s} u\right](\xi)=M_h(\xi)^{2s}\F[u](\xi)$. This is consistent with the previous definitions, as pointed out in Lemma \ref{l:commute}.

The symbol $M_h(\xi)^{2s}$ is a second-order approximation for $|\xi|^{2s}$, as by Taylor expansion one has
\begin{equation}\label{e:secorderapprox}
M_h(\xi)^{2s}=|\xi|^{2s}\left(1+O(h^2|\xi|^2)\right).
\end{equation}
So one can hope that the finite difference given by $(-\Delta_h)^{s}$ has accuracy $h^2$.

The scheme for $s\le1$ (but general $d$) has already been studied in \cite{HZD21} (and the special case $d=1$ already in \cite[Section 4.2]{HO16} and, in more detail, in \cite{CRSTV18}) and has many desirable properties. First of all, for $s\in\N$, we recover the standard schemes for polyharmonic Laplacians. We also have the property that $(-\Delta_h)^s(-\Delta_h)^{s'}=(-\Delta_h)^{s+s'}$. Moreover, while for other schemes the accuracy often degenerates as $s\nearrow1$, our scheme has accuracy $h^2$ uniformly in $s$ (as follows from Theorem \ref{t:error_estimate}). In Remark \ref{r:findiffscheme} below, we comment on how this scheme might work in practice.

Now that we have chosen our scheme, let us discuss other rigorous estimates for its approximation quality. In the literature on finite difference schemes, it is common to derive pointwise estimates on the error under a strong regularity assumption ($C^k$ or $C^{k,\,\alpha}$ for a large enough $k$). In fact, for $d\in\{1,2\}$ and $s\le1$, there are two such results in the literature: in \cite{CRSTV18}, pointwise estimates for the approximation error for functions in Hölder spaces (at least $C^{0,\,2s+\varepsilon}$) are shown and, in \cite{HZD21}, such pointwise estimates are shown under the assumption that the $(2s+\varepsilon)$-th derivative has integrable Fourier transform (which, roughly speaking, again corresponds to $C^{0,\,2s+\varepsilon}$).

However, as already mentioned in the previous subsection, our interest is more in estimates under low-regularity assumptions, i.e., in a Sobolev scale. To the best of our knowledge, such estimates are new (even in the case $d=1$, $s< 1$). For the case of the Laplacian or Bilaplacian, though, such results are classical (see the textbook \cite[Chapter 2]{JS14} or a recent refinement for the Bilaplacian in \cite{S20}), and our result is inspired by the latter. However, the proof is quite different. Namely,  the proof of \cite[Theorem 2.3]{S20} relied on the Bramble--Hilbert lemma and thereby used that $s\in\N$. In our general setting, we use a different approach, based on the Poisson summation formula and a lengthy estimate of various error terms in Fourier space.

\subsection{Main results}
\label{sec:main}

Let us now state our main results more precisely. We consider the discrete FGF $\varphi_h$, formally defined as 
\begin{align*}
\PP_h(\mathrm{d}\varphi_h)=\frac{1}{Z_h}\exp\left(-\frac12\sum_{x\in\Omega_h}h^d\varphi_h(x)((-\Delta_h)^s\varphi_h)(x)\right)\dd\varphi_h, 
\end{align*}
 and the continuous FGF $\varphi$, as introduced in Section \ref{s:introFGF}. As the rigorous definitions are quite technical, we postpone them to Sections \ref{ssec:discrete-FGF} and \ref{ssec:cont-FGF}, respectively. 

We claim that the scaling limit of $\varphi_h$ in an appropriate sense is $\varphi$. However, $\varphi_h$ is defined only on $h\Z^d$, so we need to interpolate it to a function on $\R^d$ first. For that purpose, we fix a compactly supported function $\Theta\in\mathcal{S}(\R^d)$ with $\int_{\R^d}\Theta(x)\dd x=1$ and define $\Theta_h(x)\coloneqq \frac{1}{h^d}\Theta\left(\frac{x}{h}\right)$.

Using $\Theta_h$, we can define the interpolated field
\[I_h\varphi_h(x)\coloneqq \sum_{y\in h\Z^d}h^d\varphi_h(y)\Theta_h(x-y)\]
as a random element of $\mathcal{S}'(\R^d)$. 

Some of the results below also hold if $\Theta$ is just a tempered distribution (and, in fact, in \cite{CDH19}, only the choice $\Theta=\delta_0$ was used). However, if we hope to find a scaling limit in some Banach space of optimal regularity, we need to consider more regular $\Theta$ (as otherwise $I_h\varphi_h$ might not even be an element of the Banach space in question); so, to avoid unneccessarily complicated notations, we directly assume that $\Theta$ is a measurable function.

As a first result, we claim that, for any $\Theta$ chosen as above, the interpolated fields $I_h\varphi_h$ converge in the sense of distributions. We note that our definition of $\varphi_h$ is made in such a way that we do not need to rescale it with some power of $h$ to obtain a scaling limit. Indeed, we have the following result.

\begin{theorem}[Scaling limit in the space of distributions]\label{t:scalinglimit_distributions}
	Let $\Omega\subset\R^d$ be a bounded domain with Lipschitz boundary, and let $s\ge0$.

	Let $\Theta$ be a compactly supported function with integral 1. Then $I_h\varphi_h$ converges in law with respect to the topology of $\mathcal{S}'(\R^d)$ to $\varphi$. That is, for any $f\in\mathcal{S}(\R^d)$, the random variable $(I_{h}\varphi _{h},f)_{L^{2}(\mathbb{R}^{d})} = \int _{\mathbb{R}^{d}}
I_{h}\varphi _{h} (x) f(x) \,\mathrm{d}x$ converges in law to
$(\varphi ,f)_{L^{2}(\mathbb{R}^{d})} = \int _{\mathbb{R}^{d}}
\varphi (x) f(x) \,\mathrm{d}x$.
\end{theorem}

In Theorem \ref{t:scalinglimit_distributions}, we established a scaling limit in the space of distributions. However, as we discuss in detail in Section \ref{sec:FGF}, the continuous FGF is defined not just as a distribution-valued random variable, but actually has a certain Besov-,  Sobolev- and Hölder-regularity. Hence, it is natural that we can take the scaling limit of the $\varphi_h$ also in these spaces. In order to do so, however, we need some further assumptions on the regularization $\Theta$  (otherwise the interpolated field $I_h\varphi_h$ might not even be an element of the space). The result now is the following.

\begin{theorem}[Scaling limits in Sobolev and H\"older Spaces]\label{t:scalinglimit_sob_hold}
	Let $\Omega\subset\R^d$ be a bounded domain with Lipschitz boundary and let $s\ge0$. Let $\Theta$ be a compactly supported function with integral 1, and suppose that there exists some $k$, with $k>s+\frac d2$, such that
	\begin{equation}\label{e:assumption_Theta}
	|\F\Theta(\xi)|\le C\frac{(\sum_{j=1}^d\sin^2(\xi_j))^{k/2}}{|\xi|^k}
	\end{equation}
	for all $\xi$ and some $C$. 
	
	Let $s'\in\R$ with $s'<s-\frac d2$, and let $p,q\in[1,\infty]$. Then the interpolated fields $I_h\varphi_h$ converge in law with respect to the topology of $\hat B^{s'}_{p,\,q}(\R^d)$ to $\varphi$. 
	
	Moreover, for any fixed bounded domain $\hat\Omega$ with Lipschitz boundary such that $\Omega\Subset\hat\Omega$, $I_h\varphi_h$ are supported in $\hat\Omega$ for $h$ sufficiently small. For any $s'<s-\frac d2$, the interpolated fields $I_h\varphi_h$ converge in law with respect to the topology of $\dot H^{s'}(\hat \Omega)$ to $\varphi$.
	In addition, if $H\coloneqq s-\frac d2>0$, $m=\lceil H\rceil -1$, and $0<\alpha< H-m$, then $I_h\varphi_h$ converges in law with respect to the topology of $C^{m,\,\alpha}(\R^d)$ to $\varphi$.
	
\end{theorem}
Here, $\hat B^{s'}_{p,\,q}(\R^d)$ is, up to a minor technicality that we again discuss in Section \ref{sec:FGF}, equal to the standard Besov space $B^{s'}_{p,\,q}(\R^d)$.

Several remarks are in order.
First of all, a convenient example of a function satisfying \eqref{e:assumption_Theta} for some $k\in\N$ is given by the centered B-spline of order $k$ (see, e.g., \cite[Chapter 4]{Schu07} for an introduction and \cite[Section 1.9.4]{JS14} for a summary in the context of finite difference schemes).

Next, the range of parameters is basically optimal. Indeed, the continuous FGF is not in $\hat B^{s'}_{p,\,q}$ for any $s'\ge s-\frac d2$ and any $p,q$ with the only possible exception $s'=s-\frac d2$, $p<\infty$, and $q=\infty$ (cf. \cite{MR2918350}, where the case $s=0$ is studied in detail; the general case should be similar). In view of this, Theorem \ref{t:scalinglimit_sob_hold} is optimal apart from the fact that it does not cover the endpoint case $s'=s-\frac d2$, $p<\infty$, and $q=\infty$.

Regarding the convergence in Sobolev spaces, we cannot expect convergence with respect to the topology of $\dot H^{s'}(\Omega)$ for the simple reason that, because of the mollification, $I_h\varphi_h$ need not have zero boundary values outside of $\Omega$.

A fundamental step in the proof of the results above is establishing the following error estimate for a fractional Poisson equation (which is of interest in itself). Our goal is to compare the solutions of $(-\Delta)^su=f$ and of $(-\Delta_h)^su_h=f$ and we will estimate the error $u-u_h$ in the (discrete) energy norm $\|\cdot \|_{\dot H^s_h}$. We refer to Sections \ref{ssec:cont-FGF} and \ref{ssec:discrete-FGF} for the definitions of continuous and discrete fractional Sobolev norms, respectively. As we work under minimal regularity assumptions on $u$, the precise result is somewhat more technical. Namely, in general $u$ and $f$ might not be continuous functions and so it is not clear how to restrict them to the lattice. We circumvent this by introducing two additional mollifiers.

The result then takes the following shape.
\begin{theorem}[Error estimate on the discrete approximation]
	\label{t:error_estimate}
	Let $\Omega\subset\R^d$ be an open bounded set with Lipschitz boundary. Let $\Theta\colon\, \R^d\to\R$ and $\theta\colon\, \R^d\to\R$ be mollifiers that are compactly supported, symmetric around 0, and have integral 1. Furthermore, let us assume that there exist $k,\,l\ge0$ such that
	\begin{align*}
	|\F\Theta(\xi)|&\le C\frac{(\sum_{j=1}^d\sin^2(\xi_j))^{k/2}}{|\xi|^k},\\
	|\F\theta(\xi)|&\le C\frac{1}{(1+|\xi|)^l},
	\end{align*}
	for some $C$ and define $\Theta_h(x)\coloneqq \frac{1}{h^d}\Theta\left(\frac{x}{h}\right)$, $\theta_h(x)\coloneqq \frac{1}{h^d}\theta\left(\frac{x}{h}\right)$.
	
	Let $0<s<t$ and let $u\in H^t(\R^d)$ be the solution of
	\begin{align*}
	\begin{cases}
	(-\Delta)^s u(x)=f(x), &x \in \Omega,\\
	u(x)=0, & x \in \R^d\setminus\Omega,
	\end{cases}	
	\end{align*}
	for some $f\in H^{t-2s}(\R^d)$, and $u_h\colon\,h\Z^d\to\R$ be the solution of
	\begin{align*}
	\begin{cases}
	(-\Delta_h)^s u_h(x) =\Theta_h*f(x), & x \in h\Z^d\cap\Omega, \\
	u_h(x)=0, & x \in h\Z^d\setminus\Omega.
	\end{cases}	
	\end{align*}
	Then, if $h\le1$, $k\ge s$, $k>\frac{d}{2}+2s-t$, $l>\frac{d}{2}-t$, and $t-s\le 2$, we have the estimate
	\[\|\theta_h*u-u_h\|_{\dot H^s_h(h\Z^d)}\le Ch^{t-s}\|u\|_{\dot H^t(\R^d)},\]
	where $C>0$ depends on $\Theta,\,\theta,\,s,\,t,\,\Omega$, but not on $h$.
\end{theorem}
Here (as in the rest of the paper), $C$ denotes some generic constant that might change from line to line, but is always independent of $h$.

Let us give some explanations regarding the linear constraints on the parameters in this result. The most important constraint is $t-s \le 2$. It arises from the fact that the proposed finite difference scheme is of second order (see \cite[Section 5.2]{HO16}), so that the accurary of our scheme saturates at $h^2$.
The condition $k > \frac{d}{2}+2s-t$ is needed in order for $\Theta_h \ast f$ to be continuous (so that it has a well-defined restriction to the lattice). Similarly, the condition $l>\frac{d}{2}-t$ is needed in order for $\theta_h \ast u$ to be continuous.

As mentioned in Section \ref{s:findifffrac}, this is the first rigorous estimate for a finite difference scheme for $(-\Delta)^s$ under low regularity assumptions. For finite elements, a result that is similar in spirit can be found in \cite[Theorem 2.6]{BLN22}. There an estimate for piecewise linear finite elements is shown that is similar to our result (albeit with the additional restriction $t-s\le\frac12$ instead of $t-s\le2$). The method of proof is very different.

\subsection{Future work}
The most well-studied discrete Gaussian interface model is certainly the discrete Gaussian free field (corresponding to $s=1$ in our notation). In recent years, there has been a lot of activity to extend results known for the discrete Gaussian free field to other discrete (Gaussian or non-Gaussian) interface models, and this work is a first step to include the discrete FGFs in the latter class.

Let us highlight one such question, namely regarding the maximum of the field. In case of the discrete Gaussian free field, this is well-understood. In the subcritical dimension ($d=1$), the field is nothing but a random walk bridge, so it is easy to see that the rescaled maximum converges to a non-degenerate random variable. In supercritical dimensions ($d\ge3$), correlations decay so rapidly that the maximum behaves as if the field values were independent \cite{CCH16}. The most interesting case is the critical case, $d=2$, where the field is log-correlated and obtains the typical second-order correction \cite{BDZ16}.

These results have already been extended to the case of the membrane model (corresponding to $s=2$). The subcritical case ($d\le3$) was studied in \cite{CDH19} using results from \cite{MS19}, the supercritical case in \cite{CCH16}, and finally the critical case in \cite{S20}. An important tool in the latter proof was an estimate for finite difference schemes very similar to the one in Theorem \ref{t:error_estimate}.

For general $s$, it is very likely that similar results hold true. In fact, in the subcritical case $d<2s$, convergence of the rescaled maximum is a straightforward corollary of Theorem \ref{t:scalinglimit_sob_hold}.
\begin{corollary}[Convergence of the maximum for $d<2s$]\label{c:convmaxsubcrit}
Let $d<2s$, $\Omega\subset\R^d$ be a fixed bounded domain with Lipschitz boundary, and consider the family $\varphi_h$ of $s$-FGF on $\Omega_h\coloneqq \Omega\cap h\Z^d$ as $h\to0$. Then the random variables $\max_{x\in\Omega_h}\varphi_h(x)$ converge in distribution to a non-degenerate random variable.
\end{corollary}
While this corollary covers the subcritical case, the critical case ($d=2s$) and the supercritical case ($d>2s$) remain open, and we hope to address them in the future. In particular, a study of the critical case would be very interesting, as most existing examples of discrete log-correlated fields in the literature are in $d=2$ or some other even dimension while the $\frac32$-discrete FGF, for instance, would be a natural example of a log-correlated field in odd dimensions.

\section{Fractional polyharmonic Gaussian fields}
\label{sec:FGF}
In this section, we give precise definitions for the continuous and discrete FGF. For the continuous FGF, we follow \cite{LSSW16}. The major difference is that we only require Lipschitz continuity of the boundary of our domain $\Omega$ (and hence our results cover in particular the important special case $\Omega=(0,1)^d$). Because of this, several functional-analytic statements require extra attention and we give precise references for the results we use.

\subsection{The (continuous) fractional Gaussian field}
\label{ssec:cont-FGF}
We first fix our conventions for the Fourier transform, and then use it to define some relevant function spaces.

For a function $u\colon\, \R^d\to\R$, we let $\F[u](\xi)\colon\, \R^d\to\R$, defined by
\[\F[u](\xi)\coloneqq \int_{\R^d}\e^{-i\xi\cdot x}u(x)\dd x,\]
be its continuous Fourier transform. Then, we have the Fourier inversion formula,
\[u(x)=\F^{-1}[\F [u]](x)=\frac{1}{(2\pi)^d}\int_{\R^d}\e^{i\xi\cdot x}\F[u](\xi)\dd\xi,\]
and Plancherel's theorem,
\[\int_{\R^d}|u(x)|^2\dd x=\frac{1}{(2\pi)^d}\int_{\R^d}|\F[u](\xi)|^2\dd\xi.\]
We can also define the Sobolev norms
\begin{align*}
\|u\|^2_{\dot H^s(\R^d)}&\coloneqq \int_{\R^d}|\xi|^{2s}|\F[u](\xi)|^2\dd\xi,\\
\|u\|^2_{H^s(\R^d)}&\coloneqq \int_{\R^d}(1+|\xi|^2)^s|\F[u](\xi)|^2\dd\xi,
\end{align*}
where we note that $|\xi|^2$ is the Fourier multiplier of the Laplacian $-\Delta$.

Let $\Omega\subset\R^d$ be a bounded domain with Lipschitz boundary and let $\mathcal{S}'(\R^d)$ be the space of tempered distributions on $\R^d$. In what follows, we collect some results on fractional Sobolev spaces on $\Omega$. If $\Omega$ has a smooth boundary, all of them are well-known (and \cite[Chapter 4]{T78} is a comprehensive reference). If $\Omega$ has merely Lipschitz boundary, the situation is slightly more complicated and we  rely on the reference \cite{T02}. A brief version of some of these results is also contained in \cite[Section 4.1]{LSSW16}, but some of them are not made explicit there.

For $s\ge0$, let $\dot{\tilde H}^s(\Omega)$ be the closure of $C_c^\infty(\Omega)$ with respect to the norm $\|\cdot\|_{\dot H^s(\R^d)}$ \footnote{~In \cite{LSSW16} this space is denoted $H^s_0(\Omega)$. However, more commonly $\dot H^s_0(\Omega)$ is defined as the closure of $C_c^\infty(\Omega)$ with respect to the norm $\|\cdot\|_{\dot H^s(\Omega)}$, while our space $\dot{\tilde H}^s(\Omega)$ is equal to the Lions--Magenes space (which is also denoted by $\dot H^s_{00}(\Omega)$). The two spaces are different whenever $s\in\N+\frac12$. Our notation is based on the one in \cite[Chapter 4]{T78}.} and let $\dot H^{-s}(\Omega)$ be its dual space. The space $\dot H^{-s}(\Omega)$ can alternatively be described as follows. Let
\[\|u\|_{\dot H^{-s}(\Omega)}\coloneqq \inf_{\substack{v\in \dot H^{-s}(\R^d)\\u=v \text{ in }\Omega}}\|v\|_{\dot H^{-s}(\R^d)}\]
for $u\in\mathcal{S}(\R^d)$, where $H^{-s}(\R^d)$ is the dual space of $H^s(\R^d)$. Then, if $\mathcal{S}(\Omega)$ is the quotient of $\mathcal{S}(\R^d)$ under the equivalence relation that identifies functions when they agree in $\Omega$, we have that $\dot H^{-s}(\Omega)$ is the closure of $\mathcal{S}(\Omega)$ with respect to the norm $\|\cdot\|_{\dot H^{-s}(\Omega)}$ (this follows from \cite[Theorem 3.5 (i)]{T02} upon observing that our spaces $\dot{\tilde H}^s(\Omega)$, for $s\ge 0$, and $\dot H^{s}(\Omega)$, for $s<0$, are equal to Triebel's $\bar F^s_{2,\,2}(\Omega)=\bar B^s_{2,\,2}(\Omega)$ by \cite[Proposition 3.1]{T02}). Also, by \cite[Theorem 3.5 (ii)]{T02}, $\dot{\tilde H}^s(\Omega)$ for $s>0$ is equal to the space of functions in $H^s(\R^d)$ that are supported in $\bar\Omega$.

From the Lax--Milgram lemma (and the fact that $C_c^\infty(\Omega)$ is dense in $\dot{\tilde H}^s(\Omega)$) we also obtain that $(-\Delta)^s$ is an isometry from $\dot{\tilde H}^s(\Omega)$ to $\dot H^{-s}(\Omega)$.

For convenience (and with a slight abuse of notation) we define
\[\dot H^s(\Omega)\coloneqq \begin{cases}\dot{\tilde H}^s(\Omega)& \text{if } s\ge 0,\\\dot H^{s}(\Omega)& \text{if } s<0.\end{cases}\]
This scale of Hilbert spaces has various desirable properties. For any $s<t$, the embedding from $\dot H^t(\Omega)$ to $\dot H^s(\Omega)$ is compact (cf. \cite[Theorem 2.7]{T02}). Even more importantly, the spaces form an interpolation scale with respect to complex (or equivalently real) interpolation (cf. \cite[Theorem 3.5 (iv)]{T02}).

In \cite[Section 4.2]{LSSW16}, the continuous FGF is defined as a probability measure $\PP$ on $\mathcal{S}'(\R^d)$. More precisely, it is defined such that when $\varphi$ is distributed according to $\PP$, then for every Schwartz function $f\in\mathcal{S}(\R^d)$ we have that $(\varphi,f)$ is a centered Gaussian with variance $\|f\|^2_{\dot H^{-s}(\Omega)}$. By \cite[Theorem 2.3 and Proposition 2.4]{LSSW16}, this property defines $\PP$ as a probability measure on $\mathcal{S}'(\R^d)$ uniquely.

Let us remark that one can one alternatively define the FGF as a random sum of eigenfunctions of $(-\Delta)^s$. We give details on this in  Appendix \ref{app:alternative}.

The regularity of $\varphi$ is best measured in Besov spaces. For $s'\in\R$, $p,q\in[1,\infty]$, we let $\|\cdot\|_{B^{s'}_{p,\,q}(\R^d)}$ be the usual Besov norm (defined,  e.g., via Littlewood--Paley decomposition or via wavelets; see, for example, \cite[Chapter 2]{T78}) and let $\hat B^{s'}_{p,\,q}(\R^d)$ be the closure of $C_c^\infty(\R^d)$ with respect to the norm $\|\cdot\|_{B^{s'}_{p,\,q}(\R^d)}$ \footnote{~We note that the Besov space $B^{s'}_{p,\,q}(\R^d)$ is commonly defined as the set of all tempered distributions for which $\|\cdot\|_{B^{s'}_{p,\,q}(\R^d)}$ is finite. Clearly $\hat B^{s'}_{p,\,q}(\R^d)\subset B^{s'}_{p,\,q}(\R^d)$, and the inclusion is strict if $p=\infty$ or $q=\infty$.}.

Then we have the following regularity results for $\varphi$.

\begin{proposition}[Regularity of the FGF]\label{p:regularityFGF}
	Let $\Omega\subset\R^d$ be a bounded domain with Lipschitz boundary and let $s\ge0$.
	For any $s'<s-\frac d2$ and $p,q\in[1,\infty]$, the FGF on $\Omega$ is $\PP$-almost surely an element of $\hat B^{s'}_{p,\,q}(\R^d)$.
	
In particular, for any $s'<s-\frac d2$, the FGF on $\Omega$ is $\PP$-almost surely an element of $\dot H^{s'}(\Omega)$.
	
	Moreover, if $H\coloneqq s-\frac d2>0$, then the FGF on $\Omega$ is also $\PP$-almost surely an element of $C^{m,\,\alpha}_{\mathrm{loc}}(\R^d)$ for $m\coloneqq \lceil H\rceil -1$ and any $0<\alpha< H-m$. 
\end{proposition}

\begin{proof}
The Besov regularity could be shown using the tightness criterion in Lemma \ref{l:criterion_tightness} below applied to the constant sequence $\varphi^{(m)}=\varphi$. However, according to Theorem \ref{t:scalinglimit_sob_hold}, we have the much stronger statement that the FGF is the limit (with respect to the $\hat B^{s'}_{p,\,q}(\R^d)$-topology) of the discrete fractional Gaussian fields, suitably interpolated; so we do not give details for the proof of the Besov regularity here.

It is well-known that $\hat B^{s'}_{2,\,2}(\R^d)=H^{s'}(\R^d)$ and $\hat B^{s'}_{\infty,\,\infty}(\R^d)\hookrightarrow\mathcal{C}^{s'}(\R^d)$, where $\mathcal{C}^{s'}(\R^d)$ is the Hölder--Zygmund space, which embeds into the classical Hölder space $C^{\lfloor s''\rfloor, \, s''-\lfloor s''\rfloor}$ for any $0<s''<s'$ (see \cite[Section 2.1]{T02}). These results together with the fact that the FGF is supported in $\Omega$ easily imply the Sobolev and Hölder regularity results in the proposition.

Let us remark that the Sobolev regularity alternatively follows from the fact the random series defining $\tilde\varphi$ converges in $\dot H^{s'}(\Omega)$ almost surely, while the Hölder regularity also follows from \cite[Proposition 6.2 and Theorem 8.3]{LSSW16}\footnote{~N.B. There is a typo in the statement of \cite[Proposition 6.2]{LSSW16}: it should read $H-k$ instead of $H-\lceil H\rceil$.}.
\end{proof}

\subsection{The (discrete) fractional Gaussian field}
\label{ssec:discrete-FGF}
Our definition of the discrete FGF follows the one of the continuous FGF as closely as possibly. Let us again begin by fixing our conventions for discrete Fourier transforms and discrete function spaces.

For a function $u_h\colon\,h\Z^d\to\R$, we let $\F_h[u_h](\xi)\colon\, \R^d\to\R$, defined by
\[\F_h[u_h](\xi)\coloneqq h^d\sum_{x\in h\Z^d}\e^{-i\xi\cdot x}u_h(x),\]
be its discrete Fourier transform (we note that this function is $\frac{2\pi}{h}$-periodic). Then we have the discrete Fourier inversion formula,
\[u_h(x)=\F_h^{-1}[\F_h[u_h]](x)=\frac{1}{(2\pi)^d}\int_{\left(-\frac{\pi}{h},\frac{\pi}{h}\right)^d}\e^{i\xi\cdot x}\F_h[u_h](\xi)\dd\xi,\]
and Plancherel's theorem,
\[h^d\sum_{x\in h\Z^d}|u_h(x)|^2=\frac{1}{(2\pi)^d}\int_{\left(-\frac{\pi}{h},\frac{\pi}{h}\right)^d}|\F_h[u_h](\xi)|^2.\]
We can define the discrete Sobolev norms
\begin{align*}
\|u_h\|^2_{\dot H^s_h(h\Z^d)}&\coloneqq \int_{\left(-\frac{\pi}{h},\frac{\pi}{h}\right)^d}M_h(\xi)^{2s}|\F_h[u_h](\xi)|^2\dd\xi,\\
\|u_h\|^2_{H^s_h(h\Z^d)}&\coloneqq \int_{\left(-\frac{\pi}{h},\frac{\pi}{h}\right)^d}(1+M_h(\xi)^2)^s|\F_h[u_h](\xi)|^2\dd\xi,
\end{align*}
where
\(M_h(\xi)^2\coloneqq \sum_{j=1}^d\frac{4}{h^2}\sin^2\left(\frac{\xi_j h}{2}\right)\)
is the discrete Fourier multiplier of the discrete Laplacian. We remark that, for $s=0$, we recover the space $L^2_h$.

Let $\Omega$ be as before and let $\Omega_h=\Omega\cap h\Z^d$.  Similarly as in the continuous setting, we define the space $\dot{\tilde H}^s_h(\Omega_h)$ as the space of functions $h\Z^d\to\R$ that vanish outside of $\Omega_h$ (equipped with the norm induced by $\|\cdot\|_{\dot H^s_h(h\Z^d)}$). We let $\dot H^{-s}_h(\Omega_h)$ be its dual space, and define
\[\dot H^s(\Omega_h)\coloneqq \begin{cases}\dot{\tilde H}^s_h(\Omega_h)& \text{if } s\ge 0,\\\dot H^{s}_h(\Omega_h)& \text{if } s<0.\end{cases}\]

We define the discrete FGF as a probability measure on $\dot{\tilde H}^s(\Omega_h)$. More precisely, we consider the measure
\begin{align*}
\PP_h(\dd\varphi_h)&=\frac{1}{Z_h}\exp\left(-\frac12\|\varphi_h\|_{\dot H^s_h(h\Z^d)}^2\right)\prod_{x\in\Omega_h}\dd\varphi_h(x)\prod_{x\in h\Z^d\setminus\Omega_h}\delta_0(\dd\varphi_h(x))\\
&=\frac{1}{Z_h}\exp\left(-\frac12\sum_{x\in \Omega_h}h^d\varphi_h(x)(-\Delta_h)^s\varphi_h(x)\right)\prod_{x\in\Omega_h}\dd\varphi_h(x)\prod_{x\in h\Z^d\setminus\Omega_h}\delta_0(\dd\varphi_h(x)).
\end{align*}
This is a well-defined Gaussian measure with mean 0 and variance
\begin{align}\label{eq:var-sf}
\E_h(\varphi_h,f_h)_{L^2_h(h\Z^d)}^2=\|f_h\|_{\dot H^{-s}_h(\Omega_h)}^2
\end{align}
for any $f_h\colon\, h\Z^d\to\R$.

Indeed, \eqref{eq:var-sf} follows from the fact that, if $X\sim\mathcal{N}(0,\Sigma)$ is a multivariate Gaussian, then $\E(X,v)^2=(v,\Sigma^{-1}v)$, together with the observation that, if $u_h\in \dot H^s_h(\Omega_h)$ is such that $(-\Delta_h)^su_h=f_h$ in $\Omega_h$ (and $u_h$ is 0 in $(h\Z^d)\setminus \Omega_h$), then
\[(f_h,u_h)_{L^2_h(\Omega_h)}=(f_h,u_h)_{L^2_h(h\Z^d)}=\|u_h\|_{\dot H^s_h(h\Z^d)}^2=\|u_h\|_{\dot H^s_h(\Omega_h)}^2=\|(-\Delta_h)^s u_h\|_{\dot H^{-s}_h(\Omega_h)}^2=\|f_h\|_{\dot H^{-s}_h(\Omega_h)}^2.\]
This calculation used the fact that $(-\Delta_h)^s$ is an isomorphism from $ H^s_h(\Omega_h)$ to $\dot H^{-s}_h(\Omega_h)$.
\begin{remark}
For $s=1$ and $s=2$, respectively, this agrees (up to a possible rescaling of the lattice) with the standard definitions for the Gaussian free field \cite{S07} and the membrane model \cite{CDH19,S20} in the literature. For $0<s<1$, as already mentioned in the introduction, our definition is not the same as the one in the survey \cite[Section 12]{LSSW16}
\footnote{\label{fn:LSSW}~We note that there are a few small errors in \cite[Section 12.2]{LSSW16}. In particular, with the definition of a discrete FGF given there, \cite[Proposition 12.2]{LSSW16} is false. The correct definition should have the density 
\begin{equation}\label{eq:FGF_LSSW}
\exp\left(-\frac12\,\sum_{\substack{x,\,y\in\delta\Z^d\\ x\neq y}}C_{d,s}\frac{|f(x)-f(y)|^2}{|x-y|^{d+2s}}\delta^{2d}\right)
\end{equation} while, in \cite[Section 12.2]{LSSW16}, $\delta^d$ is used in place of $\delta^{2d}$. Only the definition \eqref{eq:FGF_LSSW} has the property that if we send $\delta\to0$, we formally get the density
\[\exp\left(-\frac12\int_{\R^d}\int_{\R^d} C_{d,s}\frac{|f(x)-f(y)|^2}{|x-y|^{d+2s}}\dd x\dd y\right)\]
of the continuous FGF.

The error in the proof of \cite[Proposition 12.2]{LSSW16} is in (12.10). The process $(Y^\delta_t)_{t=0}^\infty$ converges to $(Y_t)_{t=0}^\infty$ pathwise, and so the occupation measure of $A\cap \delta\Z^d$ of the former process converges to the occupation measure of $A$ of the latter. Thus, the additional factor of $\delta^d$ on the left-hand side of (12.10) is erroneous.

With this error corrected, the proof of \cite[Proposition 12.2]{LSSW16} does show that the discrete FGF with density \eqref{eq:FGF_LSSW} converges in the sense of distributions.}
. Their definition would correspond to the operator $\widetilde{(-\Delta_h)^s}$, where $\widetilde{(-\Delta_h)^s}u_h(x)=-C_{d,s}\sum_{y\in h\Z^d\setminus\{0\}}h^d\frac{u_h(y)-u_h(x)}{|x-y|^{d+2s}}$. From the point of view of numerical analysis, this discretization is quite degenerate (cf. the discussion at the beginning of Section 4.4 in \cite{HO16}). 
 
Nonetheless our methods apply to other discretizations of the fractional FGF as well, including the one in \cite{LSSW16}. The basis for all our arguments is  \eqref{e:secorderapprox}. If, instead of $O(h^2|\xi|^2)$, one only had $O(h^\kappa|\xi|^\kappa)$ for some $0<\kappa\le2$, then Theorem \ref{t:error_estimate} would hold with the restriction $t-s\le \kappa$. For the application towards Theorems \ref{t:scalinglimit_distributions} and \ref{t:scalinglimit_sob_hold}, such an estimate for some $\kappa>0$ is good enough. So our results on the scaling limit of the discrete FGF also hold for other models as long as we can show a version of \eqref{e:secorderapprox} with an error term $O(h^\kappa|\xi|^\kappa)$. That is, we need good control over the asymptotics of the symbol of the corresponding operator. In general, it can be non-trivial to obtain these asymptotics. However, for the case of $\widetilde{(-\Delta_h)^s}$, or more generally for schemes of the form $-C_{d,s}\sum_{y\in h\Z^d\setminus\{0\}}h^d k(|x-y|)(u_h(y)-u_h(x))$ for $k(x)=\frac{1}{|x|^{d+2s}}\left(1+O\left(\frac{1}{|x|^{\kappa'}}\right)\right)$ for $0<s<1$, the relevant calculations can be done as in \cite[Appendix]{H11}. In particular, also for the version of the discrete FGF from \cite{LSSW16} we have not just Theorem \ref{t:scalinglimit_distributions} (as shown in \cite[Proposition 12.2]{LSSW16}), but also Theorem \ref{t:scalinglimit_sob_hold}. This essentially answers the question in \cite[Remark 6]{G23}.
\end{remark}

\begin{remark}[Boundary values]\label{rk:prob}
Let us comment on our choice of boundary values. The main advantage of our definition is the fact that it is consistent with projections. Namely, let $\Omega\subset\tilde\Omega$ be open sets  and consider 
the discrete FGF $\tilde\varphi_h$ on $\tilde\Omega_h$. Then,  the restriction of $\tilde\varphi_h$ to $\Omega_h$ is equal in distribution to the sum of the $(-\Delta_h)^s$-harmonic extension of $\tilde\varphi_h$ from $h\Z^d\setminus\Omega_h$ to $\Omega_h$ and of an independent discrete FGF $\varphi_h$ on $\Omega_h$ \footnote{~If $s=1$, this reduces to the familiar domain Markov property for the discrete Gaussian free field: the field in a subdomain is equal in distribution to the harmonic extension of its boundary values plus an independent zero-boundary field.}. In particular, even if we had started with a field without boundary values (i.e., with $\tilde\Omega=\R^d$), then looking at the field on a subset naturally leads to consider fields with zero boundary values outside that subset.
\end{remark}

\section{Rigorous estimates for the finite difference scheme}
\label{sec:estimates}

In this section, we present the proof of Theorem \ref{t:error_estimate}. As mentioned in Section \ref{s:findifffrac}, the proof of the analogous statement for $s=2$ in \cite[Theorem 2.3]{S20} was based on the Bramble--Hilbert lemma to estimate various error terms. Thus it relied on the fact that $(-\Delta_h)^2$ (and hence the finite difference scheme)  is local in that case. 

In the generic case $s\not\in\N$, however, $(-\Delta_h)^s$ is not local, and so this proof strategy can no longer be applied. Instead, we use the fact that both $(-\Delta)^s$ and $(-\Delta_h)^s$ are defined via Fourier multipliers and directly estimate all relevant error terms in Fourier space. However, this requires extra care as we need to switch from discrete Fourier space to continuous Fourier space at some point. In fact, we need a way to compare $\F_h$ and $\F$. Fortunately, the following Poisson-type summation formula enables us to do so easily.

\begin{lemma}[Poisson-type summation formula]\label{l:disc_cont_FT}
	Suppose that $g\colon\, \R^d\to\R$ is a Schwartz function. Then we have the identity
	\[\F_h[g](\xi)=\sum_{\zeta\in\frac{2\pi}{h}\Z^d}\F [g](\xi+\zeta).\]
\end{lemma}
\begin{proof}
	By the Poisson summation formula (see, e.g., \cite[Chapter  4.4]{MR2485091}), for any Schwartz function $f$, we have
	\[h^d\sum_{x\in h\Z^d}f(x)=\sum_{\zeta\in\frac{2\pi}{h}\Z^d}\F[f](\zeta).\]
	Applying this to $f(x)=\e^{-i\xi\cdot x}g(x)$, we find
	\[h^d\sum_{x\in h\Z^d}\e^{-i\xi\cdot x}g(x)=\sum_{\zeta\in\frac{2\pi}{h}\Z^d}\F[\e^{-i\xi\cdot }g](\zeta)=\sum_{\zeta\in\frac{2\pi}{h}\Z^d}\F [g](\xi+\zeta),\]
which implies the claim.
\end{proof}

Using Lemma \ref{l:disc_cont_FT}, we can now turn to the proof of our estimate on finite difference schemes.
\begin{proof}[Proof of Theorem \ref{t:error_estimate}]
As $u\in H^t(\Omega)$ is supported in $\bar\Omega$, by the discussion in subsection \ref{ssec:cont-FGF} we can approximate it in $\dot H^t(\Omega)$-norm by functions in $C_c^\infty(\Omega)$. So by a density argument it suffices to consider the case that $u\in C_c^\infty(\Omega)$. Then, in particular, $u$ is  a Schwartz function and $\F[u]$ is a Schwartz function as well. Therefore, all integrals and sums below will be well-defined.

\textbf{Step 1.} \emph{Representation of the error.} From the definitions we have
\begin{equation}\label{e:error_estimate}
\|\theta_h \ast u-u_h\|_{\dot H^s_h(h\Z^d)}=\|(-\Delta_h)^s(\theta_h*u-u_h)\|_{\dot H^{-s}_h(\Omega_h)}=\inf_{\substack{v_h\colon\,h\Z^d\to\R\\v_h=(-\Delta_h)^s(\theta_h*u-u_h)\text{ in }\Omega_h}}\|v_h\|_{\dot H^{-s}_h(h\Z^d)}.
\end{equation}
Using Lemma \ref{l:commute}, we can also rewrite, for $x\in\Omega_h$,
\begin{align*}
&(-\Delta_h)^s(\theta_h*u-u_h)(x)\\
&\quad=(-\Delta_h)^s (\theta_h*u)(x)-\Theta_h*f(x)\\
&\quad=(-\Delta_h)^s (\theta_h*u)(x)-\Theta_h*(-\Delta)^su(x)\\
&\quad=\frac{1}{(2\pi)^d}\int_{\left(-\frac{\pi}{h},\frac{\pi}{h}\right)^d}\e^{i\xi\cdot x}M_h(\xi)^{2s}\F_h[\theta_h*u](\xi)\dd\xi-\frac{1}{(2\pi)^d}\int_{\R^d}\e^{i\xi\cdot x}|\xi|^{2s}\F[\Theta_h](\xi)\F[u](\xi)\dd\xi\\
&\quad=I_1+I_2+I_3+I_4+I_5,
\end{align*}
where
\begin{align*}
I_1(x)&\coloneqq \frac{1}{(2\pi)^d}\int_{\left(-\frac{\pi}{h},\frac{\pi}{h}\right)^d}\e^{i\xi\cdot x}M_h(\xi)^{2s}\left(\F_h[\theta_h*u](\xi)-\F[\theta_h*u](\xi)\right)\dd\xi,\\
I_2(x)&\coloneqq \frac{1}{(2\pi)^d}\int_{\left(-\frac{\pi}{h},\frac{\pi}{h}\right)^d}\e^{i\xi\cdot x}M_h(\xi)^{2s}\left(\F[\theta_h*u](\xi)-\F[u](\xi)\right)\dd\xi,\\
I_3(x)&\coloneqq \frac{1}{(2\pi)^d}\int_{\left(-\frac{\pi}{h},\frac{\pi}{h}\right)^d}\e^{i\xi\cdot x}M_h(\xi)^{2s}\left(1-\F[\Theta_h](\xi)\right)\F[u](\xi)\dd\xi,\\
I_4(x)&\coloneqq \frac{1}{(2\pi)^d}\int_{\left(-\frac{\pi}{h},\frac{\pi}{h}\right)^d}\e^{i\xi\cdot x}\left(M_h(\xi)^{2s}-|\xi|^{2s}\right)\F[\Theta_h](\xi)\F[u](\xi)\dd\xi,\\
I_5(x)&\coloneqq -\frac{1}{(2\pi)^d}\int_{\R^d\setminus\left(-\frac{\pi}{h},\frac{\pi}{h}\right)^d}\e^{i\xi\cdot x}|\xi|^{2s}\F[\Theta_h](\xi)\F[u](\xi)\dd\xi.
\end{align*}
We can choose $v_h=I_1+I_2+I_3+I_4+I_5$ in \eqref{e:error_estimate}, and so it suffices to show that
\[\|I_j\|_{\dot H^{-s}_h(h\Z^d)}\le Ch^{t-s}\|u\|_{\dot H^t(\R^d)}\]
holds for each $j\in\{1,2,3,4,5\}$. The cases $j\in\{2,3,4\}$ are easier and we begin with those.

\textbf{Step 2.} \emph{Estimate of $I_2$.} Directly from the definition, we see that
\[\F_h[I_2](\xi)=M_h(\xi)^{2s}\left(\F[\theta_h*u](\xi)-\F[u](\xi)\right)=M_h(\xi)^{2s}\left(\F[\theta_h](\xi)-1\right)\F[u](\xi).\]
First, we note that $0\le M_h(\xi) \le C|\xi|$ for some constant $C>0$. The assumptions on $\theta$ imply that $\F[\theta](0)=1$ and $\nabla\F[\theta](0)=0$. Furthermore, $\F[\theta]$ is a Schwartz function. So, by Taylor's theorem, there exists a constant $C>0$ such that $|1-\F[\theta](\xi)|\le C|\xi|^2$. Since $\mathcal{F}\left[\theta_{h}\right](\xi)=\mathcal{F}[\theta](h \xi)$, this implies $|1-\F[\theta_h](\xi)|\le Ch^2|\xi|^2$.   Therefore, 
\begin{align*}
\|I_2\|^2_{\dot H^{-s}_h(h\Z^d)}&=\int_{\left(-\frac{\pi}{h},\frac{\pi}{h}\right)^d}M_h(\xi)^{-2s}|\F_h[I_2](\xi)|^2\dd\xi\\
&=\int_{\left(-\frac{\pi}{h},\frac{\pi}{h}\right)^d}M_h(\xi)^{2s}\left|\F[\theta_h](\xi)-1\right|^2|\F[u](\xi)|^2\dd\xi\\
&\le C\int_{\left(-\frac{\pi}{h},\frac{\pi}{h}\right)^d}|\xi|^{2s}h^4|\xi|^4|\F[u](\xi)|^2\dd\xi\\
&\le C\int_{\left(-\frac{\pi}{h},\frac{\pi}{h}\right)^d}|\xi|^{2t}h^{2(t-s)}|\F[u](\xi)|^2\dd\xi\\
&\le C\int_{\left(-\frac{\pi}{h},\frac{\pi}{h}\right)^d}|\xi|^{2t}h^{2(t-s)}|\F[u](\xi)|^2\dd\xi\\
&\le Ch^{2(t-s)}\|u\|^2_{\dot H^t(\R^d)}.
\end{align*}
Here we used that $t-s\le 2$ and thus $h^4|\xi|^4\le Ch^{2(t-s)}|\xi|^{2(t-s)}$.

\textbf{Step 3.} \emph{Estimate of $I_3$.}  The estimate of $I_3$ is quite similar: again, we have that
\[\F_h[I_3](\xi)=\left(1-\F[\Theta_h](\xi)\right)M_h(\xi)^{2s}\F[u](\xi).\]
The assumptions on $\Theta$ imply that $|1-\F[\Theta](\xi)|\le C|\xi|^2$ and hence $|1-\F[\Theta_h](\xi)|\le Ch^2|\xi|^2$. Using these estimates, we can proceed exactly as in Step 2.

\textbf{Step 4.} \emph{Estimate of $I_4$.} The argument for $I_4$ is very similar to that for $I_2$ and $I_3$: we use the fact that $\left|M_h(\xi)^{2s}-|\xi|^{2s}\right|\le Ch^2|\xi|^2$ and that $|\F[\Theta_h](\xi)|\le C$ and proceed as for $I_3$.

\textbf{Step 5.} \emph{Estimate of $I_1$.}  Here,  we need to compare $\F$ and $\F_h$. Fortunately, we can use Lemma \ref{l:disc_cont_FT} for that purpose. From the definition and Lemma \ref{l:disc_cont_FT}, we have that
\begin{align*}\F_h[I_1](\xi)&=M_h(\xi)^{2s}\left(\F_h[\theta_h*u](\xi)-\F[\theta_h*u](\xi)\right)\\
&=M_h(\xi)^{2s}\sum_{\zeta\in\frac{2\pi}{h}\Z^d\setminus\{0\}}\F[\theta_h*u](\xi+\zeta)\\
&=M_h(\xi)^{2s}\sum_{\zeta\in\frac{2\pi}{h}\Z^d\setminus\{0\}}\F[\theta_h](\xi+\zeta)\F[u](\xi+\zeta).
\end{align*}
Cauchy--Schwarz' inequality then yields 
\begin{align*}
\|I_1\|^2_{\dot H^{-s}_h(h\Z^d)}&=\int_{\left(-\frac{\pi}{h},\frac{\pi}{h}\right)^d}M_h(\xi)^{-2s}|\F_h[I_1](\xi)|^2\dd\xi\\
&=\int_{\left(-\frac{\pi}{h},\frac{\pi}{h}\right)^d}M_h(\xi)^{2s}\left|\sum_{\zeta\in\frac{2\pi}{h}\Z^d\setminus\{0\}}\F[\theta_h](\xi+\zeta)\F[u](\xi+\zeta)\right|^2\dd\xi\\
&\le\int_{\left(-\frac{\pi}{h},\frac{\pi}{h}\right)^d}M_h(\xi)^{2s}\left(\sum_{\zeta\in\frac{2\pi}{h}\Z^d\setminus\{0\}}|\xi+\zeta|^{2t}|\F[u](\xi+\zeta)|^2\right)\left(\sum_{\zeta\in\frac{2\pi}{h}\Z^d\setminus\{0\}}\frac{|\F[\theta_h](\xi+\zeta)|^2}{|\xi+\zeta|^{2t}}\right)\dd\xi.
\end{align*}
We know that $\F[\theta_h](\xi+\zeta)\le\frac{C}{h^l|\xi+\zeta|^l}$. As $2(t+l)>d$, we can bound
\begin{align*}
\sup_{\xi\in \left(-\frac{\pi}{h},\frac{\pi}{h}\right)^d}\sum_{\zeta\in\frac{2\pi}{h}\Z^d\setminus\{0\}}\frac{|\F[\theta_h](\xi+\zeta)|^2}{|\xi+\zeta|^{2t}}&\le C\sup_{\xi\in \left(-\frac{\pi}{h},\frac{\pi}{h}\right)^d}\sum_{\zeta\in\frac{2\pi}{h}\Z^d\setminus\{0\}}\frac{1}{h^{2l}|\xi+\zeta|^{2(t+l)}}\\
&\le C\sum_{\zeta\in\frac{2\pi}{h}\Z^d\setminus\{0\}}\frac{1}{h^{2l}|\zeta|^{2(t+l)}}\\
&\le Ch^{2t}
\end{align*}
and deduce
\begin{align*}
\|I_1\|^2_{\dot H^{-s}_h(h\Z^d)}&\le C\frac{1}{h^{2s}}h^{2t}\int_{\left(-\frac{\pi}{h},\frac{\pi}{h}\right)^d}\sum_{\zeta\in\frac{2\pi}{h}\Z^d\setminus\{0\}}|\xi+\zeta|^{2t}|\F[u](\xi+\zeta)|^2\dd\xi\\
&\le Ch^{2(t-s)}\int_{\R^d}|\xi|^{2t}\F[u](\xi)|^2\dd\xi\\
&\le Ch^{2(t-s)}\|u\|^2_{\dot H^t(\R^d)}.
\end{align*}

\textbf{Step 6.} \emph{Estimate of $I_5$.} The argument is similar to the previous step. We see that
\[\F [I_5](\xi)=-|\xi|^{2s}\F[\Theta_h](\xi)\F[u](\xi)\chi_{\R^d\setminus\left(-\frac{\pi}{h},\frac{\pi}{h}\right)^d}(\xi),\]
where $\chi_A$ is the indicator function of the set $A$. Lemma \ref{l:disc_cont_FT} then implies that, for $\xi\in\left(-\frac{\pi}{h},\frac{\pi}{h}\right)^d$,
\begin{align*}
\F_h[I_5](\xi)&=-\sum_{\zeta\in\frac{2\pi}{h}\Z^d}\F [I_5](\xi+\zeta)\\
&=-\sum_{\zeta\in\frac{2\pi}{h}\Z^d}|\xi+\zeta|^{2s}\F[\Theta_h](\xi+\zeta)\F[u](\xi+\zeta)\chi_{\R^d\setminus\left(-\frac{\pi}{h},\frac{\pi}{h}\right)^d}(\xi+\zeta)\\
&=-\sum_{\zeta\in\frac{2\pi}{h}\Z^d\setminus\{0\}}|\xi+\zeta|^{2s}\F[\Theta_h](\xi+\zeta)\F[u](\xi+\zeta)
\end{align*}
and therefore (recalling that $M_h(\xi)$ is $\frac{2\pi}{h}$-periodic)
\begin{align*}
&\|I_5\|^2_{\dot H^{-s}_h(h\Z^d)}\\
&\quad=\int_{\left(-\frac{\pi}{h},\frac{\pi}{h}\right)^d}M_h(\xi)^{-2s}|\F_h[I_5](\xi)|^2\dd\xi\\
&\quad=\int_{\left(-\frac{\pi}{h},\frac{\pi}{h}\right)^d}M_h(\xi)^{-2s}\left|\sum_{\zeta\in\frac{2\pi}{h}\Z^d\setminus\{0\}}|\xi+\zeta|^{2s}\F[\Theta_h](\xi+\zeta)\F[u](\xi+\zeta)\right|^2\dd\xi\\
&\quad=\int_{\left(-\frac{\pi}{h},\frac{\pi}{h}\right)^d}\left|\sum_{\zeta\in\frac{2\pi}{h}\Z^d\setminus\{0\}}\frac{|\xi+\zeta|^{2s}}{M_h(\xi+\zeta)^s}\F[\Theta_h](\xi+\zeta)\F[u](\xi+\zeta)\right|^2\dd\xi\\
&\quad\le\int_{\left(-\frac{\pi}{h},\frac{\pi}{h}\right)^d}\left(\sum_{\zeta\in\frac{2\pi}{h}\Z^d\setminus\{0\}}|\xi+\zeta|^{2t}|\F[u](\xi+\zeta)|^2\right)\left(\sum_{\zeta\in\frac{2\pi}{h}\Z^d\setminus\{0\}}\frac{|\F[\Theta_h](\xi+\zeta)|^2}{M_h(\xi+\zeta)^{2s}|\xi+\zeta|^{2(t-2s)}}\right)\dd\xi.
\end{align*}
We note that $\F[\Theta_h](\xi)=\F[\Theta](h\xi)$ and so $|\F[\Theta_h](\xi)|\le C\frac{(\sum_{j=1}^d\sin^2(h\xi_j))^{k/2}}{h^k|\xi|^k}$; moreover, $\frac{(\sum_{j=1}^d\sin^2(h\xi_j))^{1/2}}{M_h(\xi)}\le Ch$. Since  $k\ge s$, $M_h(\xi+\zeta)^{2s}$ is controlled by the $\sin$-terms from $|\F[\Theta_h](\xi+\zeta)|^2$ and so we can bound
\begin{align*}
\sup_{\xi\in \left(-\frac{\pi}{h},\frac{\pi}{h}\right)^d}\sum_{\zeta\in\frac{2\pi}{h}\Z^d\setminus\{0\}}\frac{|\F[\Theta_h](\xi+\zeta)|^2}{M_h(\xi+\zeta)^{2s}|\xi+\zeta|^{2t-4s}}&\le C\sup_{\xi\in \left(-\frac{\pi}{h},\frac{\pi}{h}\right)^d}\sum_{\zeta\in\frac{2\pi}{h}\Z^d\setminus\{0\}}\frac{h^{2s}}{h^{2k}|\xi+\zeta|^{2(t+k-2s)}}\\
&\le C\sum_{\zeta\in\frac{2\pi}{h}\Z^d\setminus\{0\}}\frac{1}{h^{2(k-s)}|\zeta|^{2(t+k-2s)}}\\
&\le Ch^{2(t-s)},
\end{align*}
where we used the fact that $2(t+k-2s)>d$. Hence, 
\begin{align*}
\|I_5\|^2_{\dot H^{-s}_h(h\Z^d)}&\le Ch^{2(t-s)}\int_{\left(-\frac{\pi}{h},\frac{\pi}{h}\right)^d}\sum_{\zeta\in\frac{2\pi}{h}\Z^d\setminus\{0\}}|\xi+\zeta|^{2t}|\F[u](\xi+\zeta)|^2\dd\xi\\
&\le Ch^{2(t-s)}\int_{\R^d}|\xi|^{2t}|\F[u](\xi)|^2\dd\xi\\
&\le Ch^{2(t-s)}\|u\|^2_{\dot H^t(\R^d)}.
\end{align*}
This completes the proof.
\end{proof}

\begin{remark}[Usage of the finite difference scheme]\label{r:findiffscheme}
So far we have not said much regarding the practical applications of the finite difference scheme in Theorem \ref{t:error_estimate}. It would go beyond the scope of this work to report on some practical experiments, but let us make a few comments.

In order to use the scheme to approximate a solution of $(-\Delta)^su=f$, a first challenge is to compute the entries of $((-\Delta_h)^s)_{x,\,y\in\Omega_h}$. Even using the translation-invariance of $(-\Delta_h)^s$, we need to compute $O\left(\frac{1}{h^2}\right)$ entries, where each is given as a singular integral. This is quite costly, but avoids introducing an additional error. In fact, the pictures in Figure \ref{fig:ex} were produced using this method.

If one is willing to accept an additional error term, then a more efficient way to compute an approximation to the entries of $((-\Delta_h)^s)_{x,\,y\in\Omega_h}$ was suggested in \cite{HZD21}: choose a parameter $h'\le h$, and approximate the integral over $\left(-\frac{\pi}{h},\frac{\pi}{h}\right)^d$ appearing in the definition of $(-\Delta_h)^s$ by a Riemann sum on a lattice of width $\frac{h'}{h}$. The advantage is that this Riemann sum can be computed very efficiently using the fast Fourier transform. Moreover, in \cite[Section 4.2]{HZD21}, it is suggested that this should lead to an additional error of order $O\left(\frac{h'^{d+2s}}{h^{2s}}\right)$. In other words, if we choose $h'\le h^{(2s+2)/(2s+d)}$, the error should be of order $h^2$ and thus not bigger than the error in Theorem \ref{t:error_estimate}. While the error estimate in \cite[Section 4.2]{HZD21} is not rigorous, it should be possible to give a full proof.
\end{remark}

\section{Proofs of the scaling limits}
\label{sec:proof-m}

\subsection{Scaling limit in the space of distributions} 
\label{ssec:proof-dis}

With Theorem \ref{t:error_estimate} in hand, we are ready to prove that $\varphi$ is indeed the scaling limit of the $\varphi_h$. First, we study the scaling limit in the space of distributions, Theorem \ref{t:scalinglimit_distributions}.

\begin{proof}[Proof of Theorem \ref{t:scalinglimit_distributions}]
	\textbf{Step 1.} \emph{Characterization of the convergence.} Let us consider some $f\in\mathcal{S}(\R^d)$. Both $(I_h\varphi_h,f)_{L^2(\R^d)}$ and $(\varphi,f)_{L^2(\R^d)}$ are centered Gaussian random variables and so it suffices to prove that their variances converge. We have that
	\begin{equation}\label{e:scalinglimit_distributions2}
	\begin{split}
	\E_h(I_h\varphi_h,f)^2&=\E_h\left(\int_{\R^d}\sum_{y\in h\Z^d}h^d\varphi_h(y)\Theta_h(x-y)f(x)\dd x\right)^2\\
	&=\E_h\left(\sum_{y\in h\Z^d}h^d\varphi_h(y)\int_{\R^d}\Theta_h(x-y)f(x)\dd x\right)^2\\
	&=\E_h(\varphi_h,\Theta_h*f)_{L^2_h(h\Z^d)}^2\\
	&=\|\Theta_h*f\|_{\dot H^{-s}_h(\Omega_h)}^2
	\end{split}
	\end{equation}
	and so we only need to prove that
	\begin{equation}\label{e:scalinglimit_distributions1}
	\lim_{h\to\infty}\|\Theta_h*f\|_{\dot H^{-s}_h(\Omega_h)}^2=\|f\|_{\dot H^{-s}(\Omega)}^2.
	\end{equation}

	\textbf{Step 2.} \emph{Representation of the error.}
	For each $h>0$, let $u_h\colon\,h\Z^d\to\R$ be the solution of
	\begin{align*}
	\begin{cases}
	(-\Delta_h)^s u_h(x) =\Theta_h*f(x), & x \in h\Z^d\cap\Omega, \\
	u_h(x)=0, & x \in h\Z^d\setminus\Omega,
	\end{cases}	
	\end{align*}
	and let $u\in H^s(\R^d)$ be the solution of
	\begin{align*}
	\begin{cases}
	(-\Delta)^s u(x)=f(x), &x \in \Omega,\\
	u(x)=0, & x \in \R^d\setminus\Omega.
	\end{cases}	
	\end{align*}
	Moreover, let $\tilde\Theta,\,\tilde\theta$ be functions satisfying the assumptions of Theorem \ref{t:error_estimate} with $k\coloneqq \max\left(\frac d2+s,s\right)$ and $l\coloneqq \max\left(\frac d2-s,0\right)$; for example, let us take $\tilde\Theta$ to be a B-spline of order $\lceil k\rceil$ as in \cite[Section 1.9.4]{JS14} and $\tilde\theta$ any smooth mollifier. Then, let us define $\tilde\Theta_h$ and $\tilde\theta_h$ as before and let $\tilde u_h\colon\,h\Z^d\to\R$ be the solution of
	\begin{align*}
	\begin{cases}
	(-\Delta_h)^s \tilde u_h(x) =\tilde\Theta_h*f(x), & x \in h\Z^d\cap\Omega, \\
	\tilde u_h(x)=0, & x \in h\Z^d\setminus\Omega.
	\end{cases}	
	\end{align*}
	
	Then, we can write
	\begin{equation}\label{e:scalinglimit_distributions3}
	\begin{split}
	\|\Theta_h*f\|_{\dot H^{-s}_h(\Omega_h)}^2-\|f\|_{\dot H^{-s}(\Omega)}^2&=(\Theta_h*f,u_h)_{L^2_h(h\Z^d)}-(f,u)_{L^2(\R^d)}\\
	&=J_1+J_2+J_3+J_4+J_5,
	\end{split}
	\end{equation}
	where
	\begin{align*}
	J_1&\coloneqq (\Theta_h*f,u_h)_{L^2_h(h\Z^d)}-(\Theta_h*f,\tilde u_h)_{L^2_h(h\Z^d)},\\
	J_2&\coloneqq (\Theta_h*f,\tilde u_h)_{L^2_h(h\Z^d)}-(\Theta_h*f,\tilde\theta_h*u)_{L^2_h(h\Z^d)},\\
	J_3&\coloneqq (\Theta_h*f,\tilde\theta_h*u)_{L^2_h(h\Z^d)}-(f,\tilde\theta_h*u)_{L^2_h(h\Z^d)},\\
	J_4&\coloneqq (f,\tilde\theta_h*u)_{L^2_h(h\Z^d)}-(f,\tilde\theta_h*u)_{L^2(\R^d)},\\
	J_5&\coloneqq (f,\tilde\theta_h*u)_{L^2(\R^d)}-(f,u)_{L^2(\R^d)}.
	\end{align*}
	We need to show that $J_i \to 0$ as $h\to0$. This implies \eqref{e:scalinglimit_distributions1}, as required. The most important term is $J_2$, for which we need to use Theorem \ref{t:error_estimate}; the other terms are straightforward to control.
	
		\textbf{Step 3.} \emph{Estimate of $J_2$.}   Let $t>s$ be a constant to be chosen later. Our choices of $k,l$ ensure that the assumptions of Theorem \ref{t:error_estimate} are all satisfied. Theorem \ref{t:error_estimate} and the discrete Poincaré inequality (see Lemma \ref{l:poincare}) then imply that
	\begin{align*}
	J_2&=(\Theta_h*f,\tilde u_h-\tilde\theta_h*u)_{L^2_h(h\Z^d)}\\
	&\le\|\Theta_h*f\|_{L^2_h(h\Z^d)}\|\tilde u_h-\tilde\theta_h*u\|_{L^2_h(h\Z^d)}\\
	&\le C\|\Theta_h*f\|_{L^2_h(h\Z^d)}\|\tilde u_h-\tilde\theta_h*u\|_{\dot H^s_h(h\Z^d)}\\
	&\le C\|f\|_{L^\infty(\R^d)}h^{t-s}\|u\|_{\dot H^t(\R^d)}.
	\end{align*}
	For $t-s$ small enough (depending on $s$ and $\Omega$), Lemma \ref{l:higher_reg} implies that we have $H^{t}$-regularity-estimates on $\Omega$ and hence in particular $\|u\|_{\dot H^t(\R^d)}<\infty$. Thus $J_2\to 0$ as $h\to0$.
	
		\textbf{Step 3.} \emph{Estimate of $J_1$, $J_3$, $J_4$, $J_5$.} For $J_1$, using again the discrete Poincaré inequality, we estimate
	\begin{align*}
	J_1&=(\Theta_h*f,u_h-\tilde u_h)_{L^2_h(h\Z^d)}\\
	&\le\|\Theta_h*f\|_{L^2_h(h\Z)^d}\|u_h-\tilde u_h\|_{L^2_h(h\Z^d)}\\
	&\le C\|f\|_{L^\infty(\R^d)}\|u_h-\tilde u_h\|_{\dot H^s_h(\Omega_h)}\\
	&\le C\|f\|_{L^\infty(\R^d)}\|\Theta_h*f-\tilde\Theta_h*f\|_{\dot H^{-s}_h(\Omega_h)}\\
	&\le C\|f\|_{L^\infty(\R^d)}\|\Theta_h*f-\tilde\Theta_h*f\|_{L^2_h(\Omega_h)}\\
	&\le Ch\|f\|_{L^\infty(\R^d)}\|\nabla f\|_{L^\infty(\R^d)},
	\end{align*}
	where the right-hand side tends to $0$ as $h\to0$. The same argument also applies to $J_3$.
	
	For $J_5$, it suffices to observe that $\tilde\theta_h*u$ tends to $u$ in $L^2(\R^d)$. Finally, for $J_4$, we use the fact that $\tilde\theta_h*u$ is continuous (and thus $f\cdot(\tilde\theta_h*u)$ is continuous), and so
	\[\lim_{h\to0}(f,\tilde\theta_h*u)_{L^2_h(h\Z^d)}=(f,\tilde\theta_h*u)_{L^2(\R^d)}\]
	as a Riemann sum.

\end{proof}

\subsection{Scaling limit in Besov, Sobolev and H\"older spaces}
\label{ssec:proof-sh}

We now turn to the proof of the scaling limit in Besov spaces (which then implies the result in Sobolev and H\"older spaces as well). As we have already established convergence of the fields in the space of distributions, the main challenge is to prove tightness in Besov spaces. To this end, we use a very convenient criterion from \cite{FM17}. As in our case we do not need to worry about boundary issues, we do not need the full generality of that criterion. Let us state the version that we will use.

\begin{lemma}[Tightness criterion]\label{l:criterion_tightness}
Let $r\in\N$ and let $\hat\Omega\subset\R^d$ be an open bounded set. Then there exist functions $f,(g_j)_{j=1}^{2^d-1}\in C_c^r(\R^d)$ such that, for any multi-index $\underline{m}\in\N^d$ with $|m|<r$ and any $j\in\{1,\ldots,2^d-1\}$, we have 
\begin{equation}\label{e:moments_g}
\int_{\R^d}x^{\underline{m}}g_j(x)\dd x=0
\end{equation}
 and such that the following statement holds. Let $(\phi_n)_{n\in\N}$ be a family of random linear forms on $C^r_c(\R^d)$ with support in $\hat\Omega$. Let $t,t'\in\R$ with $t<t'$, $|t|,|t'|<r$ and let $p\in[1,\infty)$, $q\in[1,\infty]$. Let us suppose that there exists a constant $C$ such that
\begin{equation}\label{e:criterion_tightness1}
\sup_{n\in\N}\sup_{x\in\R^d}\left(\E\left| \langle \phi_n,f(\cdot-x)\rangle \right|^p\right)^{1/p}<\infty
\end{equation}
and
\begin{equation}\label{e:criterion_tightness2}
\sup_{n\in\N}\sup_{x\in\R^d}\max_{1\le j\le 2^d-1}\left(\E\left| \langle \phi_n,g_j(2^a(\cdot-x))\rangle \right|^p\right)^{1/p}\le \frac{C}{2^{a(d+t')}}\quad\text{for all } a\in\N.
\end{equation}
Then the family $(\phi_n)_{n\in\N}$ is tight in $\hat B^t_{p,\,q}(\R^d)$. If $t<t'-\frac{d}{p}$, it is also tight in $\hat B^t_{\infty,\,q}(\R^d)$.
\end{lemma}

We note that the assumptions are independent of the parameter $q$, only the integrability $p$ and regularity $t$ are important. Also,  \eqref{e:criterion_tightness2} is required to hold for all $a\in\N$, not just for $a=0$ (as would correspond to \eqref{e:criterion_tightness1}). Here the additional assumption \eqref{e:moments_g} on the $g_j$ will come in.

\begin{proof}
This is essentially \cite[Theorem 2.30]{FM17}. There a local version of the theorem is given. The global version presented here is obtained by choosing $U=\R^d$, $\hat\Omega\subset K_1\subset K_2\subset\ldots$ such that already $K_1$ is far larger than $\hat\Omega$, $k_1=k_2=\ldots =0$ and observing that, for functions with uniformly compact support, the local and global Besov spaces agree.
Finally, the assertion that $\int_{\R^d}x^{\underline{m}}g_j(x)\dd x=0$ is stated in \cite[Equation (2.2)]{FM17}.
\end{proof}

\begin{proof}[Proof of Theorem \ref{t:scalinglimit_sob_hold}]

\textbf{Step 1.} \emph{Simplifications.} It suffices to prove tightness of $I_h\varphi_h$ in the corresponding spaces; the convergence then follows easily from Theorem \ref{t:scalinglimit_distributions} by the same argument as in \cite[Proof of Theorem 3.11]{CDH19}. In order to prove tightness, we will apply Lemma \ref{l:criterion_tightness}. We fix some open bounded set $\hat\Omega\Supset\Omega$ and note that for $h$ small enough $I_h\varphi_h$ is supported in $\hat\Omega$. Let us fix some $r\in\N$ with $r>\left|s-\frac d2\right|$ and let $f,(g_j)$ be as in the lemma. We claim that, for any $p'<\infty$, 
\begin{align}
\sup_{0<h\le 1}\sup_{x\in\R^d}\E_h\left| (I_h\varphi_h,f(\cdot-x))_{L^{p'}(\R^d)} \right|^{p'}&< \infty,\label{e:tightness_besov1}\\
\sup_{0<h\le 1}\sup_{x\in\R^d}\max_{1\le j\le 2^d-1}\E_h\left| (I_h\varphi_h,g_j(2^a(\cdot-x)))_{L^{p'}(\R^d)} \right|^{p'}&\le \frac{C}{2^{ap'(d/2+s)}}\quad\text{for all } a\in\N.\label{e:tightness_besov2}
\end{align}
Once we have verified this, Lemma \ref{l:criterion_tightness} (with $t'=s-\frac{d}{2}$) directly implies tightness in $\hat B^{s'}_{p,\,q}(\R^d)$ for any $p\in[1,\infty)$, $q\in[1,\infty]$, and, choosing $p'$ sufficiently large such that $s'<t'-\frac{d}{p'}$, we cover the case $p=\infty$ as well. Once we know tightness in Besov spaces, the tightness in Sobolev- and Hölder spaces follows directly from Besov embedding.

Regarding \eqref{e:tightness_besov1} and \eqref{e:tightness_besov2}, we can make some immediate simplifications. First of all, it suffices to check the two estimates for $p'\in2\N$ (the result for other $p'$ then follows from Jensen's inequality). In addition, as $\varphi_h$ is a Gaussian random variable, all even moments of linear functionals of $\varphi_h$ are controlled by its second moment. This means that we only need to consider $p'=2$.
That is, we actually only need to verify that
\begin{align}
\sup_{0<h\le 1}\sup_{x\in\R^d}\E_h\left| (I_h\varphi_h,f(\cdot-x))_{L^2(\R^d)} \right|^2&<\infty,\label{e:tightness_besov3}\\
\sup_{0<h\le 1}\sup_{x\in\R^d}\max_{1\le j\le 2^d-1}\E_h\left| (I_h\varphi_h,g_j(2^a(\cdot-x)))_{L^2(\R^d)} \right|^2&\le \frac{C}{2^{2a(d+t')}} =\frac{C}{2^{a(d+2s)}}\quad\text{for all } a\in\N.\label{e:tightness_besov4}
\end{align}
The estimate \eqref{e:tightness_besov4} is the crucial one. So we give its proof in detail, and then explain how to prove \eqref{e:tightness_besov3} as well.

\textbf{Step 2.} \emph{Proof of \eqref{e:tightness_besov4}.}  Let us fix some $h\le1$, $a\in\N$, $x\in\R^d$, and abbreviate $\tilde g_j^{(a)}(y)\coloneqq g_j(-2^ay)$. A computation similar to the one in \eqref{e:scalinglimit_distributions2} shows that
\begin{equation}\label{e:tightness_besov5}
\begin{split}
\E_h\left| (I_h\varphi_h,g_j(2^a(\cdot-x)))_{L^2(\R^d)} \right|^2&=\E_h\left|\int_{y\in \R^d}\sum_{z\in h\Z^d}h^d\varphi_h(z)\Theta_h(y-z)g_j(2^a(y-x))\dd y\right|^2\\
&=\E_h\left(\sum_{z\in h\Z^d}h^d\varphi_h(z)\int_{y\in \R^d}\Theta_h(y-z)\tilde g^{(a)}_j(x-y)\dd y\right)^2\\
&=\E_h\left(\varphi_h,(\Theta_h*\tilde g_j^{(a)})(x-\cdot)\right)_{L^2_h(\Omega_h)}^2\\
&=\|(\Theta_h*\tilde g_j^{(a)})(x-\cdot)\|_{\dot H^{-s}_h(\Omega_h)}^2\\
&\le\|(\Theta_h*\tilde g_j^{(a)})(x-\cdot)\|_{\dot H^{-s}_h(h\Z^d)}^2.
\end{split}
\end{equation}
We estimate the right-hand side of \eqref{e:tightness_besov5} by arguing in Fourier space (similarly as in the proof of Theorem \ref{t:error_estimate}). Namely, using Lemma \ref{l:disc_cont_FT} and the fact that the Fourier transform of a convolution is the product of the Fourier transforms, we compute
\begin{align*}
&\E_h\left| (I_h\varphi_h,g_j(2^a(\cdot-x)))_{L^2(\R^d)} \right|^2\\
&\quad\le\int_{\left(-\frac{\pi}{h},\frac{\pi}{h}\right)^d}M_h(\xi)^{-2s}|\F_h[(\Theta_h*\tilde g_j^{(a)})(x-\cdot)](\xi)|^2\dd\xi\\
&\quad=\int_{\left(-\frac{\pi}{h},\frac{\pi}{h}\right)^d}M_h(\xi)^{-2s}\left|\sum_{\zeta\in\frac{2\pi}{h}\Z^d}\F[(\Theta_h*\tilde g_j^{(a)})(x-\cdot)](\xi+\zeta)\right|^2\dd\xi\\
&\quad=\int_{\left(-\frac{\pi}{h},\frac{\pi}{h}\right)^d}M_h(\xi)^{-2s}\left|\sum_{\zeta\in\frac{2\pi}{h}\Z^d}\F[(\Theta_h(x-\cdot)](\xi+\zeta)\F[\tilde g_j^{(a)}(x-\cdot)](\xi+\zeta)\right|^2\dd\xi.
\end{align*}
Next, we fix some $t$ with $\frac d2<t<k-s$ and use Cauchy--Schwarz' inequality (as in the proof of Theorem \ref{t:error_estimate}) to rewrite this as
\begin{equation}\label{e:tightness_besov6}
\begin{split}
&\E_h\left| (I_h\varphi_h,g_j(2^a(\cdot-x)))_{L^2(\R^d)} \right|^2\\
&\quad\le\int_{\left(-\frac{\pi}{h},\frac{\pi}{h}\right)^d}M_h(\xi)^{-2s}\left(\sum_{\zeta\in\frac{2\pi}{h}\Z^d}\left(\frac1h+|\xi+\zeta|\right)^{2t}|\F[\Theta_h(x-\cdot)](\xi+\zeta)|^2|\F[\tilde g_j^{(a)}(x-\cdot)](\xi+\zeta)|^2\right)\\
&\qquad\qquad\times\left(\sum_{\zeta\in\frac{2\pi}{h}\Z^d}\frac{1}{\left(\frac1h+|\xi+\zeta|\right)^{2t}}\right)\dd\xi.
\end{split}
\end{equation}
Observing that
\[\sup_{\xi\in\left(-\frac{\pi}{h},\frac{\pi}{h}\right)^d}\sum_{\zeta\in\frac{2\pi}{h}\Z^d}\frac{1}{\left(\frac1h+|\xi+\zeta|\right)^{2t}}\le C\sum_{\zeta\in\frac{2\pi}{h}\Z^d}\frac{h^{2t}}{\left(1+h|\zeta|\right)^{2t}}\le Ch^{2t}\]
as well as the fact that $M_h(\xi)$ is $\frac{2\pi}{h}$-periodic, we can rewrite \eqref{e:tightness_besov6} as
\begin{equation}\label{e:tightness_besov7}
\begin{split}
&\E_h\left| (I_h\varphi_h,g_j(2^a(\cdot-x)))_{L^2(\R^d)} \right|^2\\
&\le Ch^{2t}\int_{\left(-\frac{\pi}{h},\frac{\pi}{h}\right)^d}\sum_{\zeta\in\frac{2\pi}{h}\Z^d}M_h(\xi+\zeta)^{-2s}\left(\frac1h+|\xi+\zeta|\right)^{2t}|\F[\Theta_h(x-\cdot)](\xi+\zeta)|^2|\F[\tilde g_j^{(a)}(x-\cdot)](\xi+\zeta)|^2\dd\xi\\
&=Ch^{2t}\int_{\R^d}M_h(\xi)^{-2s}\left(\frac1h+|\xi|\right)^{2t}|\F[\Theta_h(x-\cdot)](\xi)|^2|\F[\tilde g_j^{(a)}(x-\cdot)](\xi)|^2\dd\xi.
\end{split}
\end{equation}
After all these manipulations, we have rewritten the term to be estimated as an integral involving the absolute values of the Fourier transforms of $\Theta_h$, $\tilde g_j^{(a)}$. To complete the proof, we use our assumptions on $\Theta_h$, $\tilde g_j^{(a)}$ to bound these Fourier transforms.

Regarding $\Theta_h$, we know that $\F[\Theta_h](\xi)=\F[\Theta](h\xi)$; so assumption \eqref{e:assumption_Theta} implies that $|\F[\Theta_h](\xi)|\le C\frac{(\sum_{j=1}^d\sin^2(h\xi_j))^{k/2}}{h^k|\xi|^k}$ and 
\begin{equation}\label{e:tightness_besov9}
|\F[\Theta_h(x-\cdot)](\xi)|\le C\frac{(\sum_{j=1}^d\sin^2(h\xi_j))^{k/2}}{h^k|\xi|^k}.
\end{equation}
Regarding $\tilde g_j^{(a)}$, we first note that $\F[\tilde g_j^{(a)}](\xi)=\F[g_j(-2^a\cdot)](\xi)=\frac{1}{2^{ad}}\F[g_j]\left(-\frac{\xi}{2^a}\right)$.
As $g_j\in C^r_c(\R^d)$, we know that $\F[g_j]$ is smooth and decays at least like $\frac{1}{|\xi|^r}$ as $\xi\to\infty$. On the other hand, the moments of $g_j$ up to order $r-1$ vanish by \eqref{e:moments_g} and so $\nabla^m\F[g_j](0)=0$ for any $m\le r-1$. By Taylor's theorem, this implies $|\F[g_j](\xi)|\le C|\xi|^r$. Altogether, we conclude that $|\F[g_j](\xi)|\le C\frac{|\xi|^r}{(1+|\xi|)^{2r}}$ and thus also
\[|\F[\tilde g_j^{(a)}](\xi)|\le C\frac{1}{2^{ad}}\frac{\left|\frac{\xi}{2^a}\right|^r}{\left(1+\left|\frac{\xi}{2^a}\right|\right)^{2r}}=\frac{C|\xi|^r}{2^{a(d+r)}(1+2^{-a}|\xi|)^{2r}}\]
and
\begin{equation}\label{e:tightness_besov10}
|\F[\tilde g_j^{(a)}(x-\cdot)](\xi)|\le \frac{C|\xi|^r}{2^{a(d+r)}(1+2^{-a}|\xi|)^{2r}}.
\end{equation}

Returning to \eqref{e:tightness_besov7}, we obtain that
\begin{equation}\label{e:tightness_besov8}
\begin{split}
&\E_h\left| (I_h\varphi_h,g_j(2^a(\cdot-x)))_{L^2(\R^d)} \right|^2\\
&\quad\le Ch^{2t}\int_{\R^d}\frac{h^{2s}}{\sum_{j=1}^d(\sin^2(h\xi_j))^s}\frac{(1+h|\xi|)^{2t}}{h^{2t}}\left|\frac{(\sum_{j=1}^d\sin^2(h\xi_j))^{k/2}}{h^k|\xi|^k}\right|^2\left|\frac{|\xi|^r}{2^{a(d+r)}(1+2^{-a}|\xi|)^{2r}}\right|^2\dd\xi\\
&\quad\le C\frac{h^{2(s-k)}}{2^{2a(d+r)}}\int_{\R^d}\frac{(1+h|\xi|)^{2t}(\sum_{j=1}^d\sin^2(h\xi_j))^{k-s}}{|\xi|^{2(k-r)}(1+2^{-a}|\xi|)^{4r}}\dd\xi.
\end{split}
\end{equation}
In particular, the integrand decays like $\frac{1}{|\xi|^{2(k+r-t)}}$ and so our assumptions $t<k-s$ and $r>\left|s-\frac d2\right|\ge\frac{d}{2}-s$ ensure its integrability at $\xi = \infty$. At $\xi = 0$, the integral behaves like $\frac{1}{|\xi|^{2(s-r)}}$, which  is integrable since $r>\left|s-\frac d2\right|\ge s-\frac d2$.

We distinguish the two cases whether $h<2^{-a}$ or $h\ge2^{-a}$. In the former case, we can bound the integral on the right-hand side of \eqref{e:tightness_besov8} as
\begin{align*}
&\int_{\R^d}\frac{(1+h|\xi|)^{2t}(\sum_{j=1}^d\sin^2(h\xi_j))^{k-s}}{|\xi|^{2(k-r)}(1+2^{-a}|\xi|)^{4r}}\dd\xi\\
&\quad\le C\int_{|\xi|\le 2^a}\frac{1\cdot(h|\xi|)^{2(k-s)}}{|\xi|^{2(k-r)}\cdot 1}\dd\xi+C\int_{2^a<|\xi|\le 1/h}\frac{1\cdot(h|\xi|)^{2(k-s)}}{|\xi|^{2(k-r)}\cdot(2^{-a}|\xi|)^{4r}}+C\int_{|\xi|> 1/h}\frac{(h|\xi|)^{2t}\cdot 1}{|\xi|^{2(k-r)}(2^{-a}|\xi|)^{4r}}\dd\xi\\
&\quad\le C2^{ad}h^{2(k-s)}2^{2a(r-s)}+C\frac{1}{h^d}2^{4ar}h^{2(k-s)}h^{2(r+s)}+C\frac{1}{h^d}2^{4ar}h^{2t}h^{2(k+r-t)}\\
&\quad\le C2^{a(d+2r-2s)}h^{2(k-s)}\left(1+2^{a(2r+2s-d)}h^{2r+2s-d}+2^{a(2r+2s-d)}h^{2r+2s-d}\right)\\
&\quad\le C2^{a(d+2r-2s)}h^{2(k-s)},
\end{align*}
where in the last step we used that $2^ah<1$ and $2r+2s-d>0$. 
In case $h\ge2^{-a}$, we can similarly estimate the integral by
\begin{align*}
&\int_{\R^d}\frac{(1+h|\xi|)^{2t}(\sum_{j=1}^d\sin^2(h\xi_j))^{k-s}}{|\xi|^{2(k-r)}(1+2^{-a}|\xi|)^{4r}}\dd\xi\\
&\quad\le C\int_{|\xi|\le 1/h}\frac{1\cdot(h|\xi|)^{2(k-s)}}{|\xi|^{2(k-r)}\cdot 1}\dd\xi+C\int_{1/h<|\xi|\le 2^a}\frac{(h|\xi|)^{2t}\cdot 1}{|\xi|^{2(k-r)}\cdot1}\dd\xi+C\int_{|\xi|> 2^a}\frac{(h|\xi|)^{2t}\cdot 1}{|\xi|^{2(k-r)}(2^{-a}|\xi|)^{4r}}\dd\xi\\
&\quad\le C\frac{1}{h^d} h^{2(k-s)}h^{2(s-r)}+C2^{ad}h^{2t}2^{2a(-k+r+t)}+C2^{ad}2^{4ar}h^{2t}2^{2a(-k-r+t)}\\
&\quad\le C2^{a(d+2r-2s)}h^{2(k-s)}\left(2^{a(-d-2r+2s)}h^{-d-2r+2s}+2^{2a(-k+s+t)}h^{2(-k+s+t}+2^{2a(-k+s+t)}h^{2(-k+s+t)}\right)\\
&\quad\le C2^{a(d+2r-2s)}h^{2(k-s)}.
\end{align*}
Thus, in any case, the integral on the right-hand side of \eqref{e:tightness_besov8} is bounded by $C2^{a(d+2r-2s)}h^{2(k-s)}$. Using this, \eqref{e:tightness_besov8} implies  \eqref{e:tightness_besov4}.

\textbf{Step 3.} \emph{Proof of \eqref{e:tightness_besov3}.}
The proof of \eqref{e:tightness_besov3} is similar. One difference is that we no longer need to prove decay of the term in question, only boundedness. On the other hand, the function $f$ does not satisfy a moment bound like \eqref{e:moments_g} and so we have less control over the behavior of $\F[f]$ near 0. To deal with the latter problem, we will use the Poincar\'e inequality on a suitable bounded domain $\tilde{\tilde\Omega}_h$ right in the beginning of the argument to replace the term $\frac{1}{M_h(\xi)^{2s}}$ with $\frac{1}{(1+M_h(\xi)^2)^s}$ and thereby make sure there is no singularity at $0$.

In more detail, let us fix some bounded domain $\tilde\Omega\Supset\hat\Omega$ such that $\supp(f(x-\cdot))\subset\tilde\Omega$ whenever $x\in\hat\Omega$. Then \eqref{e:tightness_besov3}  vanishes for $x\not\in\tilde\Omega$, and we can restrict attention to $x\in\tilde\Omega$. Let us fix a bounded domain $\tilde{\tilde\Omega}$ such that $\supp(\Theta_h*f(x-\cdot))\subset\tilde{\tilde\Omega}$ when $x\in\tilde\Omega$ and $h\le1$, and let $\tilde{\tilde\Omega}_h=\tilde{\tilde\Omega}\cap h\Z^d$.

We also abbreviate  $\tilde f(y)=f(-y)$. Arguing as for \eqref{e:tightness_besov4}, we find that, for $x\in\tilde\Omega$ and $h\le 1$, 
\[\E_h\left| (I_h\varphi_h,f(\cdot-x))_{L^2(\R^d)} \right|^2\le\|(\Theta_h*\tilde f)(x-\cdot)\|_{\dot H^{-s}_h(\Omega_h)}^2\le\|(\Theta_h*\tilde f)(x-\cdot)\|_{\dot H^{-s}_h(\tilde{\tilde\Omega}_h)}^2.\]
Since $\supp(\Theta_h*\tilde f(x-\cdot))\subset\tilde{\tilde\Omega}$, we use the Poincar\'e inequality (Lemma \ref{l:poincare}) to deduce that the $\|\cdot\|_{\dot H^{-s}_h(\tilde{\tilde\Omega}_h)}$-norm is bounded by a multiple of the (inhomogenous) $\|\cdot\|_{H^{-s}_h(h\Z^d)}$-norm. Indeed, if we fix another bounded domain $\Omega^*\Supset\tilde{\tilde\Omega}$ and a cut-off function $\eta\in C_c^\infty(\Omega^*)$ that is equal to 1 in $\tilde{\tilde\Omega}$ and let $\eta_h$ be its restriction to $h\Z^d$, then for any $g_h\colon h\Z^d\to\R$ that is supported in $\tilde{\tilde\Omega}_h$ we have
\[
\|g_h\|_{\dot H^{-s}_h(\tilde{\tilde\Omega}_h)}\le \|\eta_h g_h\|_{\dot H^{-s}_h(\Omega^*_h)}\le \sup_{j_h\in \dot H^s_h(\Omega^*_h)}\frac{(g_h,\eta_hj_h)_{L^2_h(h\Z^d)}}{\|j_h\|_{\dot H^s_h(\Omega^*_h)}}\le \|g_h\|_{H^{-s}_h(h\Z^d)}\sup_{j_h\in \dot H^s_h(\Omega^*_h)}\frac{\|\eta_hj_h\|_{{H^s_h(h\Z^d)}}}{\|j_h\|_{\dot H^s_h(\Omega^*_h)}}
\]
and the second factor is bounded uniformly in $h$, as Young's convolution inequality (in the form $\int |a*b|^2\le(\int |a|^2)(\int |b|)^2$) and the estimate $(1+M_h(\xi)^2)^s\le C\left((1+M_h(\xi-\zeta)^2)^s+(1+M_h(\zeta)^2)^s\right)$ for $\xi,\zeta\in\left(-\frac{\pi}{h},\frac{\pi}{h}\right)^d$  imply that
\begin{align*}
\|\eta_hj_h\|_{H^s_h(h\Z^d)}^2&=\int_{\left(-\frac{\pi}{h},\frac{\pi}{h}\right)^d}(1+M_h(\xi)^2)^s\left|\int_{\left(-\frac{\pi}{h},\frac{\pi}{h}\right)^d}\F_h[\eta_h](\zeta)\F_h[j_h](\xi-\zeta)\dd\zeta\right|^2\dd\xi\\
&\le C\int_{\left(-\frac{\pi}{h},\frac{\pi}{h}\right)^d}\left|\int_{\left(-\frac{\pi}{h},\frac{\pi}{h}\right)^d}(1+M_h(\zeta)^2)^{s/2}\F_h[\eta_h](\zeta)\F_h[j_h](\xi-\zeta)\dd\zeta\right|^2\dd\xi\\
&\qquad+C\int_{\left(-\frac{\pi}{h},\frac{\pi}{h}\right)^d}\left|\int_{\left(-\frac{\pi}{h},\frac{\pi}{h}\right)^d}(1+M_h(\xi-\zeta)^2)^{s/2}\F_h[\eta_h](\zeta)\F_h[j_h](\xi-\zeta)\dd\zeta\right|^2\dd\xi\\
&\le C\left(\int_{\left(-\frac{\pi}{h},\frac{\pi}{h}\right)^d}(1+M_h(\xi)^2)^{s/2}|\F_h[\eta_h](\xi)|^2\dd\xi\right)^2\left(\int_{\left(-\frac{\pi}{h},\frac{\pi}{h}\right)^d}|\F_h[j_h](\xi)|^2\dd\xi\right)
\\
&\qquad+C\left(\int_{\left(-\frac{\pi}{h},\frac{\pi}{h}\right)^d}|\F_h[\eta_h](\xi)|\dd\xi\right)^2\left(\int_{\left(-\frac{\pi}{h},\frac{\pi}{h}\right)^d}(1+M_h(\xi)^2)^s|\F_h[j_h](\xi)|^2\dd\xi\right)
\end{align*}
which, by the Poincar\'e inequality on $\Omega_h^*$ and the fact that $\F_h[\eta_h]$ decays faster than any polynomial (uniformly in $h$), is bounded by $C\|j_h\|_{\dot H^s_h(\Omega^*_h)}^2$.

So in fact we have
\[\|(\Theta_h*\tilde f)(x-\cdot)\|_{\dot H^{-s}_h(\tilde{\tilde\Omega}_h)}\le C\|(\Theta_h*\tilde f)(x-\cdot)\|_{H^{-s}_h(h\Z^d)}\]
Using this, we can now continue with a calculation similar as the one for \eqref{e:tightness_besov5} to obtain
\begin{align*}
&\E_h\left| (I_h\varphi_h,f(\cdot-x))_{L^2(\R^d)}\right|^2\\
&\quad\le\|(\Theta_h*\tilde f)(x-\cdot)\|_{H^{-s}_h(h\Z^d)}^2\\
&\quad=\int_{\left(-\frac{\pi}{h},\frac{\pi}{h}\right)^d}(1+M_h(\xi)^2)^{-s}\left|\sum_{\zeta\in\frac{2\pi}{h}\Z^d}\F[(\Theta_h(x-\cdot)](\xi+\zeta)\F[\tilde f(x-\cdot)](\xi+\zeta)\right|^2\dd\xi\\
&\quad\le Ch^{2t}\int_{\R^d}(1+M_h(\xi)^2)^{-s}\left(\frac1h+|\xi|\right)^{2t}|\F[(\Theta_h(x-\cdot)](\xi)|^2|\F[\tilde f(x-\cdot)](\xi)|^2\dd\xi.
\end{align*}
Using the bound \eqref{e:tightness_besov9} for $\F[\Theta_h]$ as well as the estimate 
\[|\F[\tilde f(x-\cdot)](\xi)|\le \frac{C}{(1+|\xi|)^r}\]
(the analogue of \eqref{e:tightness_besov10}), we obtain that
\[\E_h\left| (I_h\varphi_h,f(\cdot-x))_{L^2(\R^d)} \right|^2\le Ch^{2(s-k)}\int_{\R^d}\frac{(1+h|\xi|)^{2t}(\sum_{j=1}^d\sin^2(h\xi_j))^k}{(h^2+(\sum_{j=1}^d\sin^2(h\xi_j))^2)^s|\xi|^{2k}(1+|\xi|)^{2r}}\dd\xi\]
and (splitting the integral into three integrals over $|\xi|\le 1$, $1<|\xi|\le1/h$, and $1/h<|\xi|$) we see as before that the right-hand side is indeed bounded by a constant.

\end{proof}

Finally, let us give the argument for convergence of the maximum of the subcritical discrete FGF. Some technicalities arise because Theorem \ref{t:scalinglimit_sob_hold} applies to $I_h\varphi_h$ while we are interested in $\varphi_h$ itself. So we need to argue that the regularity of $I_h\varphi_h$ implies that $\varphi_h$ is necessarily close to $I_h\varphi_h$.

\begin{proof}[Proof of Corollary \ref{c:convmaxsubcrit}]
\textbf{Step 1.} \emph{Consequences of Theorem \ref{t:scalinglimit_sob_hold}.} Let $k \in \N$, with $k>s+\frac{d}{2}$, be arbitrary and take $\Theta$ to be a product of one-dimensional B-splines of order $k$ as in \cite[Section 1.9.4]{JS14}, i.e.,
\[\F[\Theta](\xi)=\prod_{j=1}^d\left(\frac{\sin(\xi)}{\xi}\right)^k.\]
This $\Theta$ is a compactly supported non-negative mollifier that satisfies the assumptions of Theorem \ref{t:scalinglimit_sob_hold}. 
Moreover, let us fix some $\alpha $ with $0<\alpha<\min\left(s-\frac{d}{2},1\right)$. 
Theorem \ref{t:scalinglimit_sob_hold} implies that $I_h\varphi_h$ converges to $\varphi$ in law with respect to the topology of $C^{0,\,\alpha}(\R^d)$. This directly implies that the maximum of $I_h\varphi_h$ converges in distribution to the maximum of $\varphi$. Therefore, it suffices to prove that the maximum of $\varphi_h$ is close enough to the maximum of $I_h\varphi_h$ in the sense that
\begin{equation}\label{e:convmax1}
\max_{x\in\Omega_h}\varphi_h(x)-\max_{y\in\R^d}I_h\varphi_h(y)\to0
\end{equation}
in probability as $h\to0$.

\textbf{Step 2.} \emph{Regularity of $\varphi_h$.} In order to prove \eqref{e:convmax1}, we need to quantify the regularity of $\varphi_h$. The idea here is that if $\varphi_h$ oscillates a lot, then also $I_h\varphi_h$ oscillates a lot and hence has large $C^{0,\,\alpha}$-norm, which is unlikely. In making this rigorous, we use our choice of $\Theta$, which simplifies some calculations.

Given an arbitrary $f_h\colon\,h\Z^d\to\R$, we claim that
\begin{equation}\label{e:convmax2}
\max_{\substack{y,\,y'\in h\Z^d\\|y-y'|_\infty\le kh}}|f_h(y)-f_h(y')|\le Ch^\alpha\|I_hf_h\|_{C^{0,\,\alpha}(\R^d)}.
\end{equation}
To see \eqref{e:convmax2}, we note that both sides are invariant under scaling the lattice, and so if we prove it for some fixed $h_*>0$, then automatically \eqref{e:convmax2} holds for all $h>0$ with a constant $c$ that does not depend on $h$. So let us fix, say, $h_*=1$.

The function $\Theta_{h_*}$ has support precisely $\left(-\frac{{h_*}k}{2},\frac{{h_*}k}{2}\right)^d$ and is piecewise a polynomial of degree at most $k-1$ in each variable.

Let us take $x\in {h_*}\Z^d$. Then $I_{h_*}f_{h_*}\restriction_{x+(0,{h_*}/2)^d}$ is a polynomial of degree at most $k-1$ in each variable, which depends precisely on the values of $f_{h_*}$ in $x+\left(-\frac{{h_*}k}{2},\frac{{h_*}(k+1)}{2}\right)^d\cap h\Z^d$. 
The space of polynomials of degree at most $k-1$ in each variable is an $\R$-vector space of dimension exactly $k^d$. The same holds true for the space of functions from $x+\left(-\frac{{h_*}k}{2},\frac{{h_*}(k+1)}{2}\right)^d\cap {h_*}\Z^d$ to $\R$.
This means that $I_{h_*}f_{h_*}$ induces a linear map between two finite-dimensional vector spaces of the same dimension. This map has trivial kernel (as follows for example from \cite[Corollary 4.62]{Schu07} together with an induction on $d$) and so it is in fact an isomorphism.

As all norms on a finite-dimensional $\R$-vector space are equivalent, we conclude that
\[\max_{y\in x+\left(-{h_*}k/2,{h_*}(k+1)/2\right)\cap {h_*}\Z^d}|f_{h_*}(y)|\le C\sup_{z\in x+(0,{h_*}/2)^d}|I_{h_*}f_{h_*}(z)|,\]
where the constant $C$ is independent of $x$ by translation-invariance.
In fact, even more is true: if $f_{h_*}$ is equal to a constant $a$ everywhere, then $I_{h_*}f_{h_*}$ is equal to the same constant $a$ (as follows for example from \cite[Theorem 4.20]{Schu07}). This implies that we actually have
\[\max_{y\in x+\left(-{h_*}k/2,{h_*}(k+1)/2\right)\cap {h_*}\Z^d}|f_{h_*}(y)-a|\le C\sup_{z\in x+(0,{h_*}/2)^d}\left|I_{h_*}f_{h_*}(z)-a\right|\]
for any $a\in\R$. By choosing $a=I_{h_*}f_{h_*}(x)$, we obtain that
\begin{align*} 
&\max_{y,\,y'\in x+\left(-{h_*}k/2,{h_*}(k+1)/2\right)\cap {h_*}\Z^d}|f_{h_*}(y)-f_{h_*}(y')|\\
&\quad\le C\max_{y,\,y'\in x+\left(-{h_*}k/2,{h_*}(k+1)/2\right)\cap {h_*}\Z^d}|f_{h_*}(y)-a|+|f_{h_*}(y')-a|\\
&\quad\le C\max_{z\in x+(0,h_*/2)^d}\left|I_{h_*}f_{h_*}(z)-I_{h_*}f_{h_*}(x)\right|\\
&\quad\le Ch_*^\alpha\|I_{h_*}f_{h_*}\|_{C^{0,\,\alpha}(x+(0,{h_*}/2)^d)}.
\end{align*} 
This shows \eqref{e:convmax2}.

We know that $I_h\varphi_h$ converges in $C^{0,\,\alpha}(\R^d)$ and so it is, in particular, tight in that space. This means that, if we define 
\[\Ev_M=\left\{[I_h\varphi_h]_{C^{0,\,\alpha}(\R^d)}\le M\right\},\] then 
$\lim_{M\to\infty}\lim_{h\to0}\PP_h(\Ev_M)=1$.
On the other hand, \eqref{e:convmax2} implies that, on the event $\Ev_M$, we have
\begin{equation}\label{e:convmax3}\max_{\substack{y,\,y'\in h\Z^d\\|y-y'|y_\infty\le kh}}|\varphi_h(y)-\varphi_h(y')|\le CMh^\alpha.
\end{equation}
This is the desired regularity estimate for $\varphi_h$.

\textbf{Step 3.} \emph{Completion of the proof.} Our specific choice of $\Theta$ has the property that $\sum_{x\in h\Z^d}h^d\Theta_h(y-x)=1$ for any $y\in\R^d$. This means that $I_h\varphi_h(y)$ is a convex combination of the $\varphi_h(x)$ with $|x-y|_\infty
<\frac{hk}{2}$, and so we have \[\max_{x\in\Omega_h}\varphi_h(x)\ge\max_{y\in\R^d}I_h\varphi_h(y),\]
which implies the lower bound in \eqref{e:convmax1}. For the upper bound we need to use \eqref{e:convmax3}. As $I_h\varphi_h(y)$ is a convex combination of the $\varphi_h(x)$ with $|x-y|_\infty
<\frac{hk}{2}$, \eqref{e:convmax3} implies that, on the event $\Ev_M$, we have $\left|\varphi_h(x)-I_h\varphi(x)\right|\le CMh^\alpha$ for any $x\in h\Z^d$.
Therefore, on the event $\Ev_M$, we have
\[\max_{x\in\Omega_h}\varphi_h(x)\le \max_{x\in\Omega_h}I_h\varphi_h(x)+CMh^\alpha\le \max_{y\in\R^d}I_h\varphi_h(y)+CMh^\alpha.\]
Putting these considerations together, we conclude that 
\[\lim_{M\to\infty}\lim_{h\to0}\PP_h\left(\max_{y\in\R^d}I_h\varphi_h(y)\le\max_{x\in\Omega_h}\varphi_h(x)\le \max_{y\in\R^d}I_h\varphi_h(y)+CMh^\alpha\right)=1,\]
which yields \eqref{e:convmax1}.
\end{proof}

\appendix

\section{Technical lemmas}
\label{sec:lemmas}

In this appendix, we provide the proof of several technical results that have been used throughout the paper.

\subsection{Discretization and restriction}

We start by proving that the applications of restricting to $h\Z^d$ and applying  $(-\Delta_h)^{s}$ commute. 

\begin{lemma}[Discretization and restriction]\label{l:commute}
	Let $u\colon\, \R^d\to\R$ be a Schwartz function. Then,  restricting to $h\Z^d$ and applying  $(-\Delta_h)^{s}$ commute: i.e., 
	\[\left((-\Delta_h)^s u\right)\re_{h\Z^d}=(-\Delta_h)^s \left(u\re_{h\Z^d}\right).\]
\end{lemma}

This allows us to be rather careless about when we restrict functions to $h\Z^d$. In fact, we will omit writing $\re_{h\Z^d}$ when (because of Lemma \ref{l:commute}) there is no ambiguity.
\begin{proof}
	The crucial fact here is that $M_h(\xi)$ is $\frac{2\pi}{h}$-periodic. Using this, we compute that, for $x\in h\Z^d$,  
	\begin{align*}
	\left((-\Delta_h)^s u\right)(x)&=\int_{\R^d}M_h(\xi)^{2s}\F[u](\xi)\dd\xi\\
	&=\sum_{\zeta\in\frac{2\pi}{h}\Z^d}\int_{\left(-\frac{\pi}{h},\frac{\pi}{h}\right)^d}M_h(\xi+\zeta)^{2s}\F[u](\xi+\zeta)\dd\xi\\
	&=\int_{\left(-\frac{\pi}{h},\frac{\pi}{h}\right)^d}M_h(\xi)^{2s}\sum_{\zeta\in\frac{2\pi}{h}\Z^d}\F[u](\xi+\zeta)\dd\xi.\\
	\end{align*}
	Using Lemma \ref{l:disc_cont_FT}, we can rewrite this as
	\begin{align*}
	\left((-\Delta_h)^s u\right)(x)&=\int_{\left(-\frac{\pi}{h},\frac{\pi}{h}\right)^d}M_h(\xi)^{2s}\F_h[u](\xi)\dd\xi\\
	&=(-\Delta_h)^s \left(u\re_{h\Z^d}\right),
	\end{align*}
	which is what we wanted to show.
\end{proof}

\subsection{Discrete inequalities}

Let us state the discrete Poincaré inequality that we used in the proof.

\begin{lemma}\label{l:poincare}
	Let $\Omega\subset\R^d$ be a bounded domain with Lipschitz boundary and let $s>0$. Then, there exists a constant $C=C(d,\Omega,s,h_*)>0$ such that, for any $h<h_*$ and any $u_h\colon\,h\Z^d\to\R$ that vanishes outside of $\Omega_h$, we have
	\[\|u_h\|_{L^2_h(\Omega_h)}\le C\|u_h\|_{\dot H^s_h(\Omega_h)}.\]
\end{lemma}
\begin{proof}
We present a discrete version of the proof in \cite[Theorem 3.7]{zbMATH07383659}. The key idea is to use Plancherel's theorem and split low and high frequencies as follows: 
\begin{align*}
\|u_h\|_{L^2_h(\Omega_h)}^2&=\int_{B_\varepsilon(0)}|\F_h{u}_h(\xi)|^2 \dd \xi+\int_{\left(-\frac{\pi}{h},\frac{\pi}{h}\right)^d \setminus B_\varepsilon(0)}|\F_h{u}_h(\xi)|^2 \dd\xi \\ &\eqqcolon I_1 + I_2,
\end{align*}
where \(\varepsilon>0\) is to be fixed later on. 

\textbf{Step 1.} \emph{Low-frequencies.} 
For the low-frequency part, $I_1$, Hölder's inequality yields
\begin{align*}
|\F_h{u}_h(\xi)| \leq\|u_h\|_{L^1_h(\Omega_h)} \leq|\Omega_h|^{1 / 2}h^{d/2}\|u_h\|_{L^2_h(\Omega_h)},
\end{align*}
where $|\Omega_h|$ denotes the cardinality of $\Omega_h$ and
\begin{align*}
\|u_h\|_{L^p_h(\Omega_h)} \coloneqq  \left(\sum_{x \in \Omega_h }h^d |u_h|^p\right)^{1/p}, \quad p \in [1,+\infty).
\end{align*}
Therefore, we have 
\begin{align*}
I_1 \le \varepsilon^d|B_1(0)||\Omega_h|h^{d}\|u_h\|_{L^2_h(\Omega_h)}^2.
\end{align*}

\textbf{Step 2.} \emph{High-frequencies.} For the high-frequency part, $I_2$, we use that $M_h(\xi)^2\ge c|\varepsilon|^2$ for $\xi\in\left(-\frac{\pi}{h},\frac{\pi}{h}\right)^d \setminus B_\varepsilon(0)$ and compute
\[
\int_{\left(-\frac{\pi}{h},\frac{\pi}{h}\right)^d \setminus B_\varepsilon(0)}|\F_h{u}_h(\xi)|^2 \dd \xi=\int_{\left(-\frac{\pi}{h},\frac{\pi}{h}\right)^d \setminus B_\varepsilon(0)} \frac{M_h(\xi)^{2s}|\F_h{u}_h(\xi)|^2}{M_h(\xi)^{2s}} \dd\xi\leq C\varepsilon^{-2 s}\|(-\Delta_h)^{s / 2} u_h\|_{L^2_h(\R^d)}^2.
\]

\textbf{Step 3.} \emph{Conclusion.}
Choosing \(0<\varepsilon<(|\Omega_h|h^d|B_1(0)|)^{-1 / d}\), we conclude
\begin{align*}
\|u_h\|_{L^2_h(\Omega_h)} \leq \frac{\varepsilon^{-s}}{\sqrt{1-\varepsilon^d|\Omega_h|h^d|B_1(0)|}}\| u_h\|_{\dot H^s_h(\Omega_h)}.
\end{align*}
By considering a square of side $L\ge \mathrm{diam}(\Omega)$ (containing $\Omega_h$), we deduce that $|\Omega_h| \le C\max\left(\frac{1}{h^d},1\right)\le \frac{C}{h^d}$ (as $h<h_*$ by assumption). This means that $|\Omega_h|h^d|B_1(0)|\le C$, and so we can make a choice of $\varepsilon>0$ independent of $h$. This concludes the proof.
\end{proof}

\begin{remark}[Generalized Poincaré inequality] Let \(s \geq t \geq 0\). Arguing as in  \cite[Theorem 1.5]{zbMATH07383659}, Lemma \ref{l:poincare} also implies that, for \(u \in \tilde H^s(\Omega)\), there exists a constant \({c}={c}(d, \Omega, s)>0\) such that
	\begin{align*}
	\|(-\Delta_h)^{t / 2} u_h\|_{L^2_h\left(\Omega_h\right)} \leq {c}\|(-\Delta_h)^{s / 2} u_h\|_{L^2_h(\Omega_h)}.
	\end{align*}
Indeed, 
	\begin{align*}
	\|(-\Delta_h)^{t / 2} u_h\|_{L^2_h(\Omega_h)}=\|u_h\|_{\dot{H}^t_h(\Omega_h)} & \leq\|u_h\|_{H^t_h(\Omega_h)} \leq\|u_h\|_{H^s_h(\Omega_h)} \leq 2^{(s+1)/2}(\|u_h\|_{L^2(\Omega_h)}+\|u_h\|_{\dot{H}^s_h(\Omega_h)}) \\
	& \leq 2^{(s+1)/2}(C\|(-\Delta_h)^{s / 2} u_h\|_{L^2(\Omega_h)}+\|(-\Delta_h)^{s / 2} u_h\|_{L^2(\Omega_h)}) \\
	&\le C\|(-\Delta_h)^{s / 2} u_h\|_{L^2(\Omega_h)}.
	\end{align*}
\end{remark} 

We also used the fact that solutions of the Dirichlet problem for $(-\Delta)^s$ have a little bit of additional regularity.

\begin{lemma}\label{l:higher_reg}
	Let $\Omega\subset\R^d$ be a bounded domain with Lipschitz boundary and let $s\ge0$. Then, there exists \(\kappa_0>0\) with the following property. If \(0 \leq \kappa \leq \kappa_0\), then, for each \(f \in H^{-s+\kappa}(\Omega)\), there exists a unique \(u \in \dot H^{s + \kappa}(\Omega) \) such that \((-\Delta)^s u=f\) in the sense of distributions; moreover, we have the estimate
		\begin{align*}
		\|u\|_{\dot H^{s+\kappa}(\Omega)} \leq C_\kappa\|f\|_{\dot H^{-s+\kappa}(\Omega)}
		\end{align*}
		for a constant \(C_\kappa\) depending only on \(\kappa\).

\end{lemma}

Let us remark that according to \cite[Theorem 2.3]{BLN22} one can take any $\kappa_0<\frac12$ here. The argument, however, is rather complicated; so we prefer to present an easy perturbative argument that gives existence of some $\kappa_0>0$ (which is enough for our purposes).

\begin{proof} We adapt the argument used in \cite[Theorem 3.3]{S20} for the biharmonic operator to the fractional case. 
	
We first show that the claimed estimate holds for $\kappa=0$. To do so, we test the equation with $u$ and deduce  
	\begin{align*}
	\|(-\Delta)^{s/2} u\|_{L^2(\Omega)}^2=\left(u, (-\Delta)^s u\right)_{L^2(\Omega)}=(u, f)_{L^2(\Omega)} \leq\|u\|_{\dot H^{s}(\Omega)}\|f\|_{\dot H^{-s}(\Omega)} .
	\end{align*}
Using Poincaré's inequality, we see that, indeed,
\[\|u\|_{\dot H^s(\Omega)} \leq C_\kappa\|f\|_{\dot H^{-s}(\Omega)}.
\]

To show that we also can take some $\kappa >0$, we use a stability result for analytic families of operators on Banach spaces. The spaces $\dot H^{s}(\Omega)$ form an interpolation family with respect to complex interpolation; thus, by \cite[Proposition 4.1]{MR961906}, the set of those $\alpha$ for which the operator \((-\Delta)^s: \dot H^{\alpha}(\Omega)\rightarrow \dot H^{\alpha-2s}(\Omega)\) has a bounded inverse is open. We have seen that this set contains $s$, so the existence of \(\kappa_0\) as in the statement of the theorem follows.
\end{proof}

\section{Fractional Gaussian Fields via eigenfunctions}
\label{app:alternative}

In this appendix, we present an alternate description of the continuous FGF. As remarked in Section \ref{ssec:cont-FGF}, $(-\Delta)^s$ is an isometry from $\dot H^s(\Omega)$ to $\dot H^{-s}(\Omega)$. Its inverse, restricted to $L^2(\Omega)$, is a positive-definitive compact operator on $L^2(\Omega)$; so, by the spectral theorem, there exists an orthonormal basis $(v_1,v_2,\ldots)$ of $L^2(\Omega)$ consisting of eigenfunctions of $(-\Delta)^s$ with associated eigenvalues $0<\lambda_1\le\lambda_2\le\ldots$. Let $X_j$ be a collection of independent standard Gaussians, and let $\tilde\varphi$ be the random variable
$\tilde\varphi=\sum_{j=1}^\infty\frac{X_j}{\sqrt{\lambda_j}}v_j$.

According to Lemma \ref{l:seriesrepresentation} below, this sum converges almost surely in $\dot H^{s'}(\Omega)\subset \dot H^{s'}(\R^d)$ for any $s'<s-\frac{d}{2}$. Therefore, $\tilde\varphi$ is a well-defined random variable on $\dot H^{s'}(\Omega)\subset \dot H^{s'}(\R^d)$. Every element of $\dot H^{s'}(\R^d)$ induces an element of $\mathcal{S}'(\R^d)$ and so we can think of $\tilde\varphi$ as a random element of $\mathcal{S}'(\R^d)$. Again, according to Lemma \ref{l:seriesrepresentation}, for any $f\in\mathcal{S}(\R^d)$ we have that $(\tilde\varphi,f)$ is a centered Gaussian with variance $\|f\|_{\dot H^{-s}(\Omega)}^2$. This means that $\tilde\varphi$ has the law $\PP$ on $\mathcal{S}'(\R^d)$ and so we can identify $\varphi$ and $\tilde\varphi$.

Let us present the aforementioned lemma.

\begin{lemma}\label{l:seriesrepresentation}
	Let $\Omega\subset\R^d$ be a bounded domain with Lipschitz boundary. Let  $s\ge0$ and $s'<s-\frac d2$ be arbitrary.
	\begin{itemize}
		\item[(i)] The series \[\tilde\varphi\coloneqq \sum_{j=1}^\infty\frac{X_j}{\sqrt{\lambda_j}}v_j\]
		converges almost surely in $\dot H^{s'}(\Omega)$.
		\item[(ii)] For any $f\in\mathcal{S}(\R^d)$, we have \[\E(\tilde\varphi,f)_{L^2(\R^d)}^2= \|f\|_{\dot H^{-s}(\Omega)}^2.\]
	\end{itemize}
\end{lemma}

For the proof, we need a sharp estimate on the eigenfunction expansion of a function.

\begin{lemma}\label{l:regularity_eigenfunction}
	Let $\Omega\subset\R^d$ be a bounded domain with Lipschitz boundary. Let $s\ge0$ and $s'\le s$ be arbitrary. Then, for any $f\in \dot H^{s'}(\Omega)$, we have
	\begin{equation}\label{e:regularity_eigenfunction}
	\|f\|_{\dot H^{s'}(\Omega)}^2\le C\sum_{j=1}^\infty \lambda_j^{s'/s}(f,v_j)_{L^2(\R^d)}^2.
	\end{equation}
\end{lemma}
We remark that we make no claim about the $\dot H^{s'}$-regularity for $s'>s$.
\begin{proof}
	For $s'\in\{-s,0,s\}$ the estimate \eqref{e:regularity_eigenfunction} follows directly from the definition. 
We next claim that \eqref{e:regularity_eigenfunction} holds whenever $-s\le s'\le s$. To see this, we adapt the argument in \cite[Corollary 1]{MR3246044}. Namely, we  take first $0<s'<s$, consider $(-\Delta)^s$ restricted to functions in $\dot H^s(\Omega)$, and let $((-\Delta)^s)_N^{s'/s}$ be its (spectral) $\frac{s'}{s}$-th power. Explicitly,
\[((-\Delta)^s)_N^{s'/s}f=\sum_{j=1}^\infty\lambda_j^{s'/s}(f,v_j)_{L^2(\R^d)}	v_j\]
We note that, if we define $\dot{\mathcal{H}}^{s'}(\Omega)$ to be the space of functions in $L^2(\Omega)$ such that this quantity is finite, then the domain of $((-\Delta)^s)_N^{s'/s}$ is exactly $\dot{\mathcal{H}}^{s'}(\Omega)$. According to the theory of interpolation of fractional powers of self-adjoint operators (see, e.g., \cite[Section 1.18.10]{T78}), the Hilbert spaces $\dot{\mathcal{H}}^{s'}(\Omega)$ form an interpolation scale. However, we know that the same holds true for the Hilbert spaces $\dot H^{s'}(\Omega)$, and moreover $\dot{\mathcal{H}}^{s'}(\Omega)=\dot H^{s'}(\Omega)$ (with equivalent norms) for $s'\in\{0,s\}$, and so we have actually have this equality for any $s'$ with $0\le s'\le s$. So, for $0\le s'\le s$, there exists some $C>0$ such that  
\[\frac1C\sum_{j=1}^\infty \lambda_j^{s'/s}(f,v_j)_{L^2(\R^d)}^2\le \|f\|_{\dot H^{s'}(\Omega)}^2\le C\sum_{j=1}^\infty \lambda_j^{s'/s}(f,v_j)_{L^2(\R^d)}^2\]
By duality, the same holds true for $-s\le s'\le0$. Putting these considerations together, we obtain a statement even stronger than \eqref{e:regularity_eigenfunction}.

It remains to study the case that $s'<-s$. We proceed inductively. Let us suppose that we know that \eqref{e:regularity_eigenfunction} holds for $s'\ge-(2k-1)s$, for some $k\in\N$, and let us consider some $s'$ with $-(2k+1)s\le s'\le -(2k-1)s$. We have that
	\[\|f\|_{\dot H^{s'}(\Omega)}^2=\inf_{\substack{g\in \dot H^{s'}(\R^d)\\f=g \text{ in }\Omega}}\|g\|_{\dot H^{s'}(\R^d)}^2.\]
	Let $u$ be such that
	\[\begin{cases}
	(-\Delta)^s u(x)=f(x), &x \in \Omega,\\
	u(x)=0, & x \in \R^d\setminus\Omega.
	\end{cases}	\]
	We can choose $g=(-\Delta)^s u$ and obtain, using the induction hypothesis, that
	\begin{align*}
	\|f\|_{\dot H^{s'}(\Omega)}^2&\le \|(-\Delta)^s u\|_{\dot H^{s'}(\R^d)}^2\\
	&\le \|u\|_{\dot H^{s'+2s}(\R^d)}^2\\
	&\le C\sum_{j=1}^\infty \lambda_j^{(s'+2s)/s}(u,v_j)_{L^2(\R^d)}^2\\
	&=C\sum_{j=1}^\infty \lambda_j^{(s'+2s)/s}\left(u,\frac{(-\Delta)^sv_j}{\lambda_j}\right)_{L^2(\R^d)}^2\\
	&=C\sum_{j=1}^\infty \lambda_j^{s'/s}\left((-\Delta)^su,v_j\right)_{L^2(\R^d)}^2\\
	&=C\sum_{j=1}^\infty \lambda_j^{s'/s}\left(f,v_j\right)_{L^2(\R^d)}^2.
	\end{align*}
	This completes the induction step.
	
\end{proof}

\begin{proof}[Proof of Lemma \ref{l:seriesrepresentation}]
	\emph{Claim (i).} By the Hilbert-space-valued version of Kolmogorov's two series theorem (see e.g. \cite[Corollary on p. 386]{GS04}), the series \[\sum_{j=1}^\infty\frac{X_j}{\sqrt{\lambda_j}}v_j\]
	converges almost surely in $\dot H^{s'}(\Omega)$ if
	\[\sum_{j=1}^\infty\left\|\frac{1}{\sqrt{\lambda_j}}u_j\right\|_{\dot H^{s'}(\Omega)}^2<\infty.\]
	From Lemma \ref{l:regularity_eigenfunction}, we know, in particular, that
	\[\|v_j\|_{\dot H^{s'}(\Omega)}^2\le \lambda_j^{s'/s}.\]
	Moreover, by Weyl's law for the operator $(-\Delta)^s$ restricted to $\dot H^s(\Omega)$ (as follows, e.g., from the main result of \cite{G14}), we have that
	\[\lambda_j\asymp j^{2s/d}.\]
	Therefore,
	\begin{align*}
	\sum_{j=1}^\infty\left\|\frac{1}{\sqrt{\lambda_j}}v_j\right\|_{\dot H^{s'}(\Omega)}^2&\le\sum_{j=1}^\infty\lambda_j^{s'/s-1}\\
	&\le\sum_{j=1}^\infty (cj)^{2s/d\cdot (s'/s-1)}\le C\sum_{j=1}^\infty j^{2(s'-s)/d},
	\end{align*}
	and this sum is indeed convergent if $s'-s<-\frac{d}{2}$. 
	
	\emph{Claim (ii).} Let $f\in\mathcal{S}(\R^d)$. The functions $v_j$ are by definition orthonormal in $L^2(\Omega)$, and the $X_j$ are independent. So we can calculate that
	\begin{align*}
	\E(\tilde\varphi,f)_{L^2(\R^d)}^2&=\E\left(\sum_{j=1}^\infty\frac{X_j}{\sqrt{\lambda_j}}(v_j,f)_{L^2(\R^d)}\right)^2\\
	&=\sum_{j=1}^\infty\frac{1}{\lambda_j}(v_j,f)_{L^2(\R^d)}^2=\left(f,\sum_{j=1}^\infty\frac{1}{\lambda_j}(v_j,f)_{L^2(\R^d)}v_j\right)\\
	&=\left(f,(-\Delta)^{-s}f\right)_{L^2(\R^d)}=\|f\|_{H^{-s}(\Omega)}^2.
	\end{align*}
\end{proof}

\vspace{5mm}
\section*{Acknowledgments}

N.~De Nitti is a member of the Gruppo Nazionale per l'Analisi Matematica, la Probabilit\`a e le loro Applicazioni (GNAMPA) of the Istituto Nazionale di Alta Matematica (INdAM).  He has been supported by the Alexander von Humboldt Foundation and by the TRR-154 project of the Deutsche Forschungsgemeinschaft (DFG, German Research Foundation).  During the revision of the manuscript, he has been supported by the Swiss State Secretariat for Education, Research and Innovation (SERI) under contract number M822.00034 through the project TENSE.

F.~Schweiger has been supported by the Foreign Postdoctoral Fellowship Program of the Israel Academy of Sciences and Humanities, and partially by ISF grant No. 421/20.

We thank O.~Zeitouni and E.~Zuazua for helpful comments on the topic of this work and W.~Ruszel for pointing out related references. We also acknowledge the anonymous referees for a careful reading of the manuscript.

\vspace{5mm}

\bibliographystyle{abbrv}
\bibliography{FGF-ref}

\vfill 

\end{document}